\documentclass[reqno]{amsart}
\usepackage{amsmath}
\usepackage{amssymb}
\usepackage{amsthm}
\usepackage{graphicx}
\usepackage{ulem}
\usepackage{epstopdf}
\usepackage{here}

\newtheorem{theorem}{Theorem}[section]
\newtheorem{proposition}[theorem]{Proposition}

\newtheorem{lemma}[theorem]{Lemma}
\theoremstyle{definition}
\newtheorem{definition}[theorem]{Definition}
\newtheorem{example}[theorem]{Example}

\newtheorem{remark}[theorem]{Remark}

\numberwithin{equation}{section}

\newcommand{\oline}[1]{\mathbin{\overline{#1}}}

\newcommand{\mincol}{{\rm mincol}^{\rm Dehn}}
\newcommand{\rank}{{\rm rank}}
\newcommand{\R}{\mathcal{R}}
\newcommand{\C}{\mathcal{C}}

\begin{document}

\title[]{Minimum numbers of Dehn colors of knots and $\mathcal{R}$-palette graphs}

\author[E.~Matsudo]{Eri Matsudo} 
\address{The Institute of Natural Sciences, Nihon University, 3-25-40 Sakurajosui, Setagaya-ku, Tokyo 156-8550, Japan}
\email{matsudo.eri@nihon-u.ac.jp}

\author[K.~Oshiro]{Kanako Oshiro}
\address{Department of Information and Communication Sciences, Sophia University, Tokyo 102-8554, Japan}
\email{oshirok@sophia.ac.jp}

\author[G.~Yamagishi]{Gaishi Yamagishi}
\address{}
\email{g-yamagishi-3c9@eagle.sophia.ac.jp}



\date{\today}

\begin{abstract}
This is the first paper which discusses minimum numbers of ``region" colors for knots, while minimum numbers of arc colors are well-studied.  
In this paper, we consider minimum numbers of colors of knots for Dehn colorings.
In particular, we will show that for any odd prime number $p$ and any Dehn $p$-colorable knot $K$, the minimum number of colors for $K$ is at least $\lfloor \log_2 p \rfloor +2$.
Moreover, we will define the $\R$-palette graph for a set of colors. The $\R$-palette graphs are quite useful to give candidates of sets of colors which might realize a nontrivially Dehn $p$-colored diagram.
In Appendix, we also prove that for Dehn $5$-colorable knot, the minimum number of colors is $4$. 
\end{abstract}

\maketitle

\section*{Introduction}





This is the first paper which discusses minimum numbers of ``region" colors for knots, while minimum numbers of arc colors are well-studied. 

In knot theory, minimum numbers of colors for arc colorings have been studied in many papers (see \cite{AbchirElhamdadiLamsifer, BentoLopes, HanZhou,  HararyKauffman, NakamuraNakanishiSatoh16, Oshiro10, Satoh09} for example). 
We denote by $\text{mincol}^{\text{Fox}}_p(K)$ the number for a Fox $p$-colorable knot $K$.
As one of properties for $\text{mincol}^{\text{Fox}}_p(K)$, the following result is obtained in \cite{NakamuraNakanishiSatoh13}.
\begin{itemize}
\item For any odd prime number $p$ and any Fox $p$-colorable knot $K$, 
$$\text{mincol}^{\text{Fox}}_p(K)\geq \lfloor \log_2 p \rfloor +2.$$
\end{itemize}
Note that for the cases when $p$ is a composite number or $K$ is a link with multiple components, the similar results are known under some condition (see \cite{IchiharaMatsudo}).  

On the other hand, a {\it Dehn $p$-coloring} is well-known as one of {\it region} colorings of knot diagrams which is corresponding to a Fox $p$-coloring.   
However, minimum numbers of colors for Dehn $p$-colorings (moreover those for region colorings) almost have not been studied yet.

In this paper, we will consider minimum numbers of colors of Dehn $p$-colorable knots for an odd prime number $p$.
The number for a Dehn $p$-colorable knot $K$ is denoted by $\mincol_p(K)$. 
We show in Theorem~\ref{th:main1} that 
\begin{itemize}
\item for any odd prime number $p$ and any Dehn $p$-colorable knot $K$, 
$$\mincol_p(K)\geq \lfloor \log_2 p \rfloor +2.$$
\end{itemize}
Moreover, we will define the {\it $\R$-palette graph} for a set of colors. The $\R$-palette graphs are quite useful to give candidates of sets of colors which might realize a nontrivially Dehn $p$-colored diagram (see Theorem~\ref{th:palette}).
In particular, in Theorem~\ref{prop:colorcandidate}, we give the sets of $\lfloor \log_2 p \rfloor +2$ colors which might realize a nontrivially Dehn $p$-colored diagram for an odd prime number $p$ with $p<2^5$. 
Furthermore, we show in Proposition~\ref{prop:pcoloreddiagram} that there exists a knot $K$ with $\mincol_p(K)= \lfloor \log_2 p \rfloor +2$ for some $p$.

This paper is organized as follows: 
In Section~\ref{sec:1}, we review the definitions and some basic properties of Dehn colorings and minimum numbers of colors for Dehn colorings.
Besides, our main result (Theorem~\ref{th:main1}) is stated in this section.
Section~\ref{sec:coloringmatrices} is devoted to setting up some extended coloring matrices which are used in Section~\ref{sec:proofofmaintheorem}. Theorem~\ref{th:main1} is proven in Section~\ref{sec:proofofmaintheorem}.
In Section~\ref{sec:palettegraph}, the $\R$-palette graph of a set of colors and that of a Dehn $p$-colored diagram are defined. 
In Section~\ref{evaluation of graph}, we give candidates of sets of colors each of which might realize a nontrivially Dehn $p$-colored diagram by using the $\R$-palette graphs.
Moreover in Appendix~\ref{appendix:5-colorableknot}, we show that $\mincol_5(K)=4$ for any Dehn $5$-colorable knot $K$. 

\section{Dehn colorings and minimum numbers of colors}\label{sec:1}
In this paper, for a prime number $p$, we denote by $\mathbb Z_p$ the cyclic group $\mathbb Z/ p\mathbb Z$.
When $p=0$, read all the parts of this paper by replacing all $\mathbb Z_p$ and $p$ with $\mathbb Z$. 

Let $p$ be an odd prime number or $p=0$.
Let $D$ be a diagram of a knot $K$ and $\mathcal{R}(D)$ the set of regions of $D$.
A Dehn {\it $p$-coloring} of $D$ is a map $C: \mathcal{R}(D) \to \mathbb Z_p$ 
satisfying the following condition: 
\begin{itemize}
\item for each crossing $c$ with regions 
$x_1, x_2, x_3$, and $x_4$ 
as depicted in Figure~\ref{coloring2},
\[
C(x_1) + C(x_3) =C(x_2) + C(x_4)
\]
holds, where the region $x_2$ is adjacent to $x_1$ by an under-arc and $x_3$ is adjacent to $x_1$ by the over-arc.
\end{itemize}
We call $C(x)$ the {\it color} of a region $x$ by $C$. 
In this paper, as shown in the right of Figure~\ref{coloring2}, we represent a Dehn $p$-coloring $C$ of a knot diagram $D$ by assigning the color $C(x)$ to each region $x$.  
We mean by $(D,C)$ a diagram $D$ given a Dehn $p$-coloring $C$, and call it a {\it Dehn $p$-colored diagram}. 
We denote by $\mathcal{C}(D, C)$ the set of colors assigned to a region of $D$ by $C$, that is $\mathcal{C}(D, C)={\rm Im}\,C$.
The set of Dehn $p$-colorings of $D$ is denoted by ${\rm Col}_{p}(D)$.
We remark that the number $\# {\rm Col}_{p}(D)$ is an invariant of the knot $K$.
\begin{figure}[ht]
  \begin{center}
    \includegraphics[clip,width=8cm]{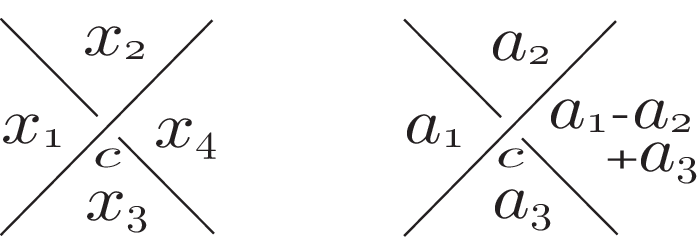}
    \caption{A crossing on $D$ and the one on $(D,C)$ with $C(x_1)=a_1$, $C(x_2)=a_2$, $C(x_3)=a_3$, $C(x_4)=a_1-a_2+a_3$}
    \label{coloring2}
  \end{center}
\end{figure}

Let $c$ be a crossing of $D$ with regions $x_1, x_2, x_3$, and $x_4$ 
as depicted in Figure~\ref{coloring2}.
We say that $c$ of $(D,C)$ is {\it trivially colored}  if 
\[
C(x_1) = C(x_4), \mbox{ and } C(x_3) =C(x_2)
\]
hold, and {\it nontrivially colored} otherwise.
 A Dehn $p$-coloring $C$ of $D$ is {\it trivial} if each crossing of $(D,C)$ is trivially colored 
 and {\it nontrivial} otherwise. 
We note that trivial $p$-colorings are separated into two types: 
monochromatic colorings and checkerboard colorings with two colors. 
We call the former {\it $1$-trivial} colorings and denote them by $C^{\rm 1T}$.
We call the latter {\it $2$-trivial} colorings and denote them by $C^{\rm 2T}$.

A knot $K$ is {\it Dehn $p$-colorable} if $K$ has a Dehn $p$-colored diagram $(D,C)$ such that $C$ is nontrivial.

For a semiarc $u$ of a Dehn $p$-colored diagram $(D,C)$, when the two regions $x_1$ and $x_2$ which are on the both side of $u$ have $C(x_1)=a_1$ and $C(x_2)=a_2$ for some $a_1,a_2\in \mathbb Z_p$, we call $u$ an {\it $\{a_1,a_2\}$-semiarc}, where $\{a_1,a_2\}$ is regarded as a multiset $\{a_1,a_1\}$ when $a_1=a_2$. 
Moreover, we call an arc $u$ including an $\{a_1,a_2\}$-semiarc an {\it $\overline{a_1+a_2}$-arc}, where we note that the value $\overline{a_1+a_2}$ does not depend on the choice of a semiarc included in $u$.

\begin{remark}\label{remark:1.1}
Fox colorings of knot diagrams have been well-studied for a long time. There is the $p$-to-$1$ correspondence, given in the upper picture of  Figure~\ref{coloring1}, from Dehn colorings to Fox colorings, where $a_1, a_2 \in \mathbb Z_p$ represent the colors of regions of a Dehn $p$-colored diagram, $\overline{a_1+a_2} \in \mathbb Z_p$ represents the colors of arcs of the corresponding Fox $p$-colored diagram, and we distinguish colors of regions for Dehn $p$-colorings and those of arcs for Fox $p$-colorings by putting a bar over a color of an arc as in Figure~\ref{coloring1}. 
Then a trivial (resp. nontrivial) Dehn $p$-coloring corresponds to a trivial  (resp. nontrivial) Fox $p$-coloring as depicted in the lower picture of  Figure~\ref{coloring1}. We also note that a knot $K$ is Dehn $p$-colorable if and only if $K$ is Fox $p$-colorable.

\begin{figure}[ht]
  \begin{center}
    \includegraphics[clip,width=7cm]{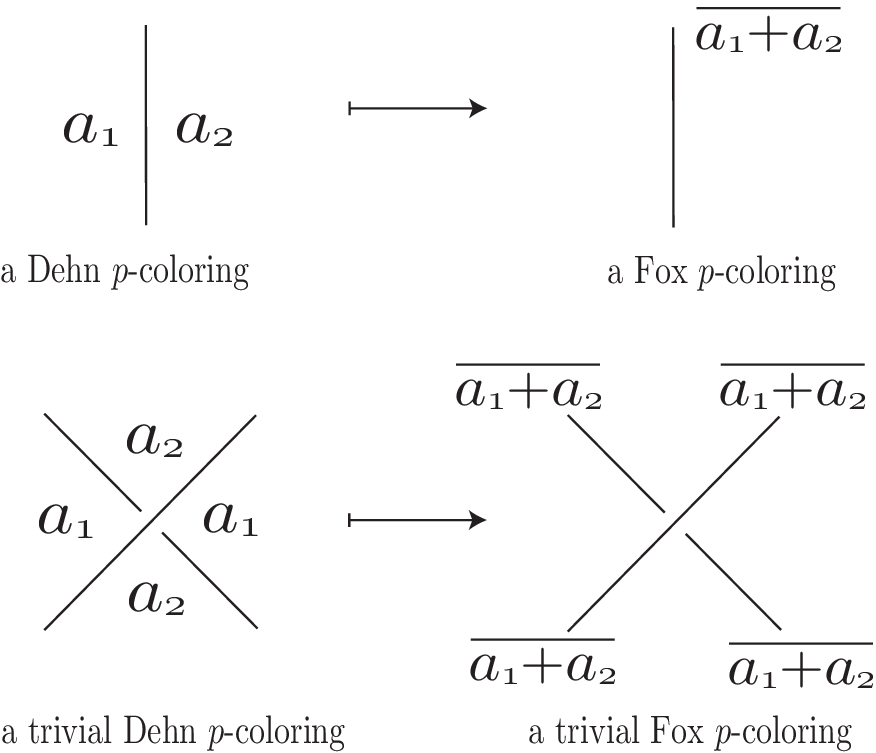}
    \caption{A correspondence between Dehn colorings and Fox colorings}
    \label{coloring1}
  \end{center}
\end{figure}
\end{remark}

The set ${\rm Col}_{p}(D)$ can be regarded as a $\mathbb Z_p$-module with the scalar product $sC$ and the addition $C+C'$ with
\begin{align*}
&(sC): \R(D) \to \mathbb Z_p; (sC)(x) = sC(x),\\ 
&(C+C'): \R(D) \to \mathbb Z_p; (C+C') (x) = C(x) + C'(x). 
\end{align*}
Moreover, since for any $t \in \mathbb Z_p$, the constant map $C^{\rm 1T}_t : \R(D) \to \mathbb Z_p; C^{\rm 1T}_t (x)=t $ is a kind of Dehn $p$-coloring, 
the map 
\[
(sC+t): \R(D) \to \mathbb Z_p; (sC+t)(x) =(sC+ C^{\rm 1T}_t)(x)= sC(x)+t,
\]
is also a Dehn $p$-coloring, and thus, the transformations $C \mapsto sC+t$ are closed in ${\rm Col}_p(D)$. 
In particular, a transformation $C \mapsto sC+t$ with $s \in \mathbb Z_p^\times$ is called a {\it regular affine transformation on ${\rm Col}_p(D)$}.
Two Dehn $p$-colorings $C, C' \in {\rm Col}_p(D)$ are {\it affine equivalent} (written $C\sim C'$) if they are related by a regular affine transformation on ${\rm Col}_p(D)$, that is, $C'=sC +t$ for some $s\in \mathbb Z_p^\times$ and $t \in \mathbb Z_p$.

A {\it regular affine transformation on the $\mathbb Z_p$-module $\mathbb Z_p$} is defined by $a\mapsto sa+t$ for some $s\in \mathbb Z_p^\times$ and $t \in \mathbb Z_p$.
Two subsets $S,S'\subset \mathbb Z_p$ are {\it affine equivalent}
(written $S\sim S'$) if they are related by a regular affine transformation on $\mathbb Z_p$, that is, $S'=sS +t$ for some $s\in \mathbb Z_p^\times$ and $t \in \mathbb Z_p$.
We note that since any regular affine transformation is a bijection, $\# S = \# S'$ holds if $S\sim S'$. 
\begin{lemma}\label{lemma:colorequiv}
\begin{itemize}
\item[(1)] Let $C, C' \in  {\rm Col}_p(D)$. 
Then we have 
\[
C \sim C' \Longrightarrow \C(D,C) \sim \C(D,C').
\]
Hence we have 
\[
C \sim C'  \Longrightarrow \#\C(D,C) = \#\C(D,C').
\]
\item[(2)] Let $S, S' \subset \mathbb Z_p$, and we assume that $S\sim S'$. Then there exists $C\in {\rm Col}_p(D)$ such that $\C(D,C) = S$ if and only if  
there exists $C'\in {\rm Col}_p(D)$ such that $\C(D,C') = S'$.
\end{itemize}
\end{lemma}
\begin{proof}
(1) A regular affine transformation $C\mapsto C'=sC+t$ on ${\rm Col}_p(D)$ induces the regular affine transformation $x\mapsto sx+t$ on $\mathbb Z_p$ with $\C(D,C') = s \C(D,C) +t$.   

(2)  Assume that $S' = s S +t$ holds for some regular affine transformation $x \mapsto sx+t$ on $\mathbb Z_p$.
When there exists $C\in {\rm Col}_p(D)$ such that $\C(D,C) = S$, 
 we have the Dehn $p$-coloring $C'=sC+t$. Then  $\C(D,C')= s \C(D,C) +t=s S +t = S'$ holds. 

\end{proof}
\begin{remark}\label{rem:colorequiv}
(1) of Lemma~\ref{lemma:colorequiv} implies that when we focus on the number of colors used for each nontrivial Dehn $p$-coloring of a diagram $D$, we may consider only the number of colors used for a representative of each affine equivalence class for nontrivial Dehn $p$-colorings of $D$.

(2) of Lemma~\ref{lemma:colorequiv} implies that when we focus on 
the sets of colors each of which might be used for a nontrivial Dehn $p$-coloring of a diagram, we may consider only representatives of the affine equivalence classes on the subsets of $\mathbb Z_p$.
\end{remark}

In this paper, we focus on the minimum number of colors.
\begin{definition} 
The {\it minimum number of colors} of a knot $K$ for Dehn $p$-colorings is the minimum number of distinct elements of $\mathbb Z_p$ which produce a nontrivially Dehn $p$-colored diagram of $K$, that is, 
\[
\min\Big\{\#\mathcal{C}(D,C) ~\Big|~ (D,C)\in \left\{
\begin{minipage}{5cm}
nontrivially Dehn $p$-colored\\
 diagrams of $K$
\end{minipage}
\right\}\Big\}.
\]
We denote it by $\mincol_p(K)$.
\end{definition}

The next theorem is our main result, which will be proven in Section~\ref{sec:proofofmaintheorem}.
\begin{theorem}\label{th:main1}
Let $p$ be an odd prime number. 
For any Dehn $p$-colorable knot $K$, we have 
\[
\mincol_p(K) \geq \log_2 p +1,
\]
that is,  
\[
\mincol_p(K) \geq \lfloor \log_2 p \rfloor +2.
\]
\end{theorem}

\section{Extended coloring matrices}\label{sec:coloringmatrices}

From now on, for an integer matrix $M$, $\rank_p M$ means the rank of $M$ regarded as a matrix on $\mathbb Z_p$. 
Note again that $\rank_0 M$ is replaced with $\rank_{\mathbb Z} M$.

Let $D$ be a diagram of a knot $K$.
Let $x_1,  \ldots , x_{n+2}$ be the regions, and $c_1, \ldots , c_n$ the crossings of $D$. 
We set an arbitrary orientation for $D$.

For each crossing $c_i$, we set the equation
\[
x_{i_1}-x_{i_2}+x_{i_3}-x_{i_4} = 0,
\]
where $x_{i_1}, x_{i_2}, x_{i_3}, x_{i_4}$ are the regions around $c_i$ such that $x_{i_1}$ is in the right side of both of over- and under-arcs,  $x_{i_2}$ is adjacent to $x_{i_1}$ by an under-arc and $x_{i_3}$ is adjacent to $x_{i_1}$ by the over-arc (see the left of Figure~\ref{coloring9}), and where we regard $x_1,  \ldots , x_{n+2}$ as variables of the equation.
We then have the simultaneous equations  
\[
\left\{
\begin{array}{l}
x_{1_1}-x_{1_2}+x_{1_3}-x_{1_4} = 0,\\
\hspace{1cm}\vdots \\
x_{i_1}-x_{i_2}+x_{i_3}-x_{i_4} = 0,\\
\hspace{1cm}\vdots \\
x_{n_1}-x_{n_2}+x_{n_3}-x_{n_4} = 0,
\end{array}
\right.
\]
that is 
\begin{align}\label{eq1}
M^{\rm col}_D 
\begin{pmatrix}
x_1\\
\vdots\\
x_{n+2}
\end{pmatrix}= \boldsymbol{0}.
\end{align}
The coefficient matrix $M^{\rm col}_D$ is called the {\it Dehn coloring matrix} of $D$ as an unoriented diagram since  
the $\mathbb Z_p$-module of solutions of (\ref{eq1}) for an odd prime $p$ or $p=0$ is isomorphic to the $\mathbb Z_p$-module ${\rm Col}_p(D)$ by the isomorphism which sends a solution $^t(a_1 \cdots a_{n+2})$ to the Dehn coloring $C$ such that $C(x_i)=a_i$ for each $i\in \{1,\ldots , n+2\}$.
\begin{figure}[ht]
  \begin{center}
    \includegraphics[clip,width=7cm]{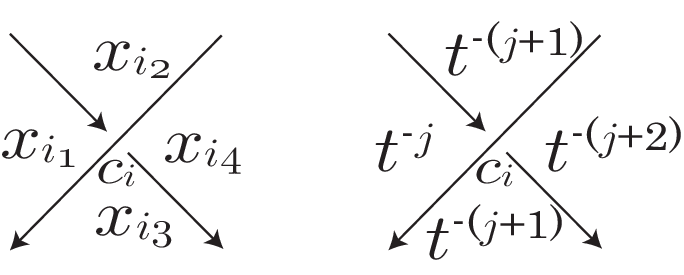}
    \caption{}
    \label{coloring9}
  \end{center}
\end{figure}

It is known that the knot group $G_K$ of $K$, which is the fundamental group of $\mathbb R^3 \setminus K$, has the presentation 
\begin{align} \label{eq2}
G_K=\langle x_1, \ldots , x_{n+2} ~|~ r_1, \ldots , r_n, x_1 \rangle,
\end{align}
where for each crossing $c_i$ $(1\leq i \leq n)$ as in the left of  Figure~\ref{coloring9}, we set the relator $r_i$ by  
\[
r_i=x_{i_3}x_{i_1}^{-1}x_{i_2}x_{i_4}^{-1}.
\]
Let $A_D(t)$ denote the Alexander matrix of $G_K$ obtained from the above presentation (\ref{eq2}) by using the Fox calculus (see \cite{CrowellFox}), that is 
$\displaystyle A_D(t) = \left( \tilde{\alpha}\circ \tilde{\pi} \left(\frac{\partial r_i}{\partial x_{j}} \right)  \right)$, 
where $\frac{\partial}{\partial x_{j}}$ is the Fox derivative $\mathbb{Z}F\rightarrow  \mathbb{Z}F$ with $F=\langle x_1, \ldots, x_{n+2}\rangle$, 
$\tilde{\pi}$ is the projection $\mathbb{Z}F\rightarrow \mathbb{Z}G_K$, and
$\tilde{\alpha}$ is the abelianization $\mathbb{Z}G_K\rightarrow \mathbb{Z}\langle t\rangle=\mathbb{Z}[t,t^{-1}]$ such that a meridian element is sent to $t$.
For the $i$th row vector of $A_D(-1)$ $(1\leq i \leq n)$, the $j$th entry is $1$ if $j=i_1$ or $i_{3}$, $-1$ if $j=i_2$ or $i_4$, and $0$ otherwise, which follows from  
\begin{align*}
&\frac{\partial r_i}{\partial x_{i_1}} = -x_{i_3} x_{i_1}^{-1} \overset{\tilde{\alpha}\circ \tilde{\pi}}{\mapsto}  -t^{-(k+1)}  t^{-k} = -t^{-2k-1}   \overset{t\mapsto -1}{\mapsto} 1 ,\\
&\frac{\partial r_i}{\partial x_{i_2}} = x_{i_3} x_{i_1}^{-1} \overset{\tilde{\alpha}\circ \tilde{\pi}}{\mapsto}  t^{-(k+1)}  t^{-k} = t^{-2k-1}   \overset{t\mapsto -1}{\mapsto} -1 ,\\
&\frac{\partial r_i}{\partial x_{i_3}} = 1 \overset{\tilde{\alpha}\circ \tilde{\pi}}{\mapsto}  1   \overset{t\mapsto -1}{\mapsto} 1 ,\\
&\frac{\partial r_i}{\partial x_{i_4}} = - x_{i_3}x_{i_1}^{-1}x_{i_2}x_{i_4}^{-1} \overset{\tilde{\alpha}\circ \tilde{\pi}}{\mapsto}  -1   \overset{t\mapsto -1}{\mapsto} -1 
\end{align*}
 (refer to the right of Figure~\ref{coloring9}), where we here suppose that $x_{i_1}, \ldots ,x_{i_4}$ differ from each other.
Thus we have  
\[
A_D(-1)=\begin{pmatrix}
&\hspace{-0.6cm}\mbox{\Large{$M_{D}^{\rm col}$}}\\
1 & \boldsymbol{0}  
\end{pmatrix}.
\]  
The greatest common divisor of all the ($n+1$)-minors of $A_D(-1)$ is called the {\it determinant} of $K$, which is denoted by ${\rm det}(K)$.

For an odd prime number $p$, we suppose that there exists a nontrivial Dehn $p$-coloring, say $C_{0}$, satisfying that for some $i, j \in \{1,\ldots , n+2\}$ such that $i<j$,  $C_{0}(x_i) = C_{0}(x_j)$, $C^{\rm 2T}(x_i) \not = C^{\rm 2T}(x_j)$ holds. 
Then we set 
\[
B_{D,C_{0}; i, j}=\begin{pmatrix}
&&\hspace{-0.3cm}\mbox{\Large{$M_{D}^{\rm col}$}}&&\\
\boldsymbol{0}  &1 &\hspace{-0.3cm} \boldsymbol{0}  &\hspace{-0.4cm} -1 & \boldsymbol{0}  
\end{pmatrix},
\]  
where for the last row vector, the $i$th and $j$th entries are $1$ and $-1$, respectively, and the other entries are $0$. 

Let us consider the ranks of the matrices $M_D^{\rm col}$, $A_D(-1)$ and $B_{D,C_0; i, j}$. 

For an odd prime $p$ or $p=0$,  we have at least two linearly independent solutions of (\ref{eq1}) on $\mathbb Z_p$: one, denoted by $\boldsymbol{a}^{\rm 1T}$, is corresponding to the 1-trivial Dehn $\mathbb Z_p$-coloring $C^{\rm 1T}$ with $C^{\rm 1T}(x_1)=1$, and the other, denoted by $\boldsymbol{a}^{\rm 2T}$, is corresponding to a 2-trivial Dehn $\mathbb Z_p$-coloring $C^{\rm 2T}$ with $C^{\rm 2T}(x_1)=0$ and  ${\rm Im}C^{\rm 2T}=\{0,1\}$. Thus we can see that $\rank_p M^{\rm col}_D \leq n$. 

Besides, it is well-known that ${\rm det}(K)\not = 0$, which gives $\rank_{\mathbb Z} A_D(-1) = n+1$, and hence, $\rank_{\mathbb Z} M_D^{\rm col} \geq n$.
Considering this property together with the fact that $\rank_{\mathbb Z}M_D^{\rm col} \leq n$, we have $\rank_{\mathbb Z}M_D^{\rm col} =n$, and indeed, the $\mathbb Z$-module of solutions of (\ref{eq1}) on $\mathbb Z$ is spanned by two vectors $\boldsymbol{a}^{\rm 1T}$ and $\boldsymbol{a}^{\rm 2T}$.

On the other hand, for an odd prime $p$, if $D$ admits a nontrivial Dehn $p$-coloring $C$, then there exists a solution $\boldsymbol{a}$ of (\ref{eq1}) on $\mathbb Z_p$ such that $\boldsymbol{a}^{\rm 1T}$,  $\boldsymbol{a}^{\rm 2T}$ and  $\boldsymbol{a}$ are linearly independent. Hence, for this case, $\rank_p M^{\rm col}_D \leq n-1$.

For the matrix $A_D(-1)$, $\rank_{\mathbb Z} A_D(-1) = n+1$ as mentioned above. 
For an odd prime $p$ or $p=0$,  
note that the $\mathbb{Z}_p$-module of solutions of the simultaneous equations 
\begin{align}\label{eq3}
A_D(-1) 
\begin{pmatrix}
x_1\\
\vdots\\
x_{n+2}
\end{pmatrix}= \boldsymbol{0}
\end{align}
on $\mathbb Z_p$ is a submodule of that of (\ref{eq1}) on $\mathbb Z_p$.
Any vector $k\boldsymbol{a}^{\rm 1T}+\ell \boldsymbol{a}^{\rm 2T}$ ($k,\ell \in \mathbb Z$, $k\not=0$) can not be a solution of  (\ref{eq3})
on $\mathbb Z_p$ since any solution $\boldsymbol{a}=\,^t(a_1 \cdots a_{n+2})$ of (\ref{eq3}) needs to satisfy $a_1=0$, while $\boldsymbol{a}^{\rm 2T}$ is a linearly independent solution of (\ref{eq3}) on $\mathbb Z_p$. 
This implies that the $\mathbb{Z}$-module of solutions of (\ref{eq3}) on $\mathbb Z$ is spanned by $\boldsymbol{a}^{\rm 2T}$. 

On the other hand, for an odd prime $p$, if $D$ admits a nontrivial Dehn $p$-coloring $C$ such that $C(x_1)=0$, then there exists a solution $\boldsymbol{a}$ of (\ref{eq3}) on $\mathbb Z_p$ such that $\boldsymbol{a}^{\rm 2T}$ and  $\boldsymbol{a}$ are linearly independent. Hence, for this case, $\rank_p A_D(-1) \leq n$.

For the matrix $B_{D,C_0; i, j}$, we note that it is defined for a given odd prime $p$, a Dehn $p$-coloring $C_0$, and $i,j \in \{1,\ldots , n+2\}$.
Note that the $\mathbb{Z}$- (resp. $\mathbb Z_p$-)module of solutions of the simultaneous equations
\begin{align}\label{eq4}
B_{D,C_0; i, j} 
\begin{pmatrix}
x_1\\
\vdots\\
x_{n+2}
\end{pmatrix}= \boldsymbol{0}
\end{align}
on $\mathbb Z$ (resp. $\mathbb Z_p$) is a submodule of that of (\ref{eq1}) on $\mathbb Z$ (resp. $\mathbb Z_p$).

Any vector $k\boldsymbol{a}^{\rm 1T}+\ell \boldsymbol{a}^{\rm 2T}$ ($k,\ell \in \mathbb Z$, $\ell\not=0$) can not be a solution of (\ref{eq4})
on $\mathbb Z$ (resp. $\mathbb Z_p$) since any solution $\boldsymbol{a}=\,^t(a_1 \cdots a_{n+2})$ of (\ref{eq4}) needs to satisfy $a_i=a_j$, while $\boldsymbol{a}^{\rm 1T}$ is a linearly independent solution of  (\ref{eq4}) on $\mathbb Z$ (resp.  $\mathbb Z_p$). 
Hence the $\mathbb Z$-module of solutions of (\ref{eq4}) on $\mathbb Z$  is spanned by $\boldsymbol{a}^{\rm 1T}$, which implies that $\rank_{\mathbb Z}B_{D,C_0; i, j} = n+1$. 

On the other hand, since $C_0$ is a nontrivial $p$-coloring, there exists a solution $\boldsymbol{a}_0=\,^t(C_0(x_1) \cdots C_0(x_{n+2}))$ of (\ref{eq4}) on $\mathbb Z_p$, and $\boldsymbol{a}^{\rm 1T}$ and  $\boldsymbol{a}_0$ are linearly independent. 
Hence, $\rank_p B_{D,C_0; i, j}\leq n$.

As a summary, we have the following properties:
\begin{lemma}\label{lem:rank}
Let $p$ be an odd prime number.
\begin{itemize}
\item[(1)] Suppose that $D$ admits a nontrivial Dehn $p$-coloring $C$.
Then we have 
$$\rank_{\mathbb Z}M_D^{\rm col} =n, \mbox{ and } \rank_p M_D^{\rm col} \leq n-1.$$ 
In particular, the module of solutions of (\ref{eq1}) on $\mathbb Z$ is spanned by $\boldsymbol{a}^{\rm 1T}$ and $\boldsymbol{a}^{\rm 2T}$, and that on $\mathbb Z_p$ is spanned by at least three linearly independent solutions: $\boldsymbol{a}^{\rm 1T}$, $\boldsymbol{a}^{\rm 2T}$ and  $\boldsymbol{a}=\,^t(C(x_1) \cdots C(x_{n+2}))$.
\item[(2)] Suppose that $D$ admits a nontrivial Dehn $p$-coloring $C$ with $C(x_1)=0$. 
Then we have 
$$\rank_{\mathbb Z} A_D(-1) =n+1, \mbox{ and } \rank_p A_{D}(-1)\leq n.$$
In particular, the module of solutions of (\ref{eq3}) on $\mathbb Z$ is spanned by $\boldsymbol{a}^{\rm 2T}$, and that on $\mathbb Z_p$ 
is spanned by at least two linearly independent solutions: $\boldsymbol{a}^{\rm 2T}$ and  $\boldsymbol{a}=\,^t(C(x_1) \cdots C(x_{n+2}))$.

\item[(3)] Suppose that $D$ admits a nontrivial Dehn $p$-coloring $C$ satisfying that $C(x_i) = C(x_j)$, $C^{\rm 2T}(x_i) \not = C^{\rm 2T}(x_j)$ for some $i, j$ with $1 \leq i<j \leq n+2$.  Then we have 
$$\rank_{\mathbb Z} B_{D,C; i, j} =n+1, \mbox{ and } \rank_p B_{D,C; i, j}\leq n.$$
In particular, the module of solutions of (\ref{eq4}) on $\mathbb Z$ is spanned by $\boldsymbol{a}^{\rm 1T}$, and that on $\mathbb Z_p$ is spanned by at least two linearly independent solutions: $\boldsymbol{a}^{\rm 1T}$ and  $\boldsymbol{a}=\,^t(C(x_1) \cdots C(x_{n+2}))$.
\end{itemize}
\end{lemma}

\section{Proof of Theorem~\ref{th:main1}}\label{sec:proofofmaintheorem}

\begin{lemma}\label{lem:2^k}
Let $M$ be an integer square matrix of order $k$ which satisfies the following condition: 
\begin{itemize}
\item[($\star$)] for each $i\in \{1,\ldots , k\}$, the multiset of nonzero entries of the $i$th row vector of $M$ is $\{-2\}$, $\{-1\}$, $\{1\}$, $\{2\}$, $\{-2,1\}$, $\{-2,2\}$, $\{-1,-1\}$, $\{-1,1\}$, $\{-1,2\}$, $\{1,1\}$, $\{-2,1,1\}$, $\{-1,-1,1\}$, $\{-1,-1,2\}$, $\{-1,1,1\}$ or $\{-1,-1,1,1\}$.
\end{itemize}

Then we have 
$$| {\rm det} M | \leq 2^k.$$
\end{lemma}
\begin{proof}
We prove this lemma by the induction of $k$ and the cofactar expantion of a matrix. In the case that $k=1$, since $|\det M|=0,1$ or $2$, it holds that $|\det M|\leq 2$.

Let $k>1$. 
We assume that any square matrix $M'$, of order $k-1$, satisfying the condition $(\star)$ has $|\det{M'}|\leq 2^{k-1}$. 

We first note that the following transformations of matrices do not affect whether the condition $(\star)$ is satisfied.
\begin{itemize}
\item Compose a matrix 
$\displaystyle \left(\boldsymbol{w}_1,\dots,\boldsymbol{w}_{j-1},-\Sigma_{\ell=1}^k\boldsymbol{w}_\ell ,\boldsymbol{w}_{j+1},\dots,\boldsymbol{w}_k\right)$ from a matrix $\left(\boldsymbol{w}_1,\dots,\boldsymbol{w}_k\right)$ satisfying the condition ($\star$).
\item Remove a row and a column from a matrix satisfying the condition ($\star$).
\end{itemize}

Put $M=\left(\boldsymbol{w}_1,\dots,\boldsymbol{w}_k\right)$.  We will show by cases.

(i) In the case that $M$ has a row whose multiset of nonzero entries is $\{-1\}$ or $\{1\}$, the cofactor expantion of the row induces $|\det M|\leq 1\cdot 2^{k-1}<2^k$.

(ii) In the case that $M$ has a row whose multiset of nonzero entries is $\{-2\}$ or $\{2\}$, the cofactor expantion of the row induces $|\det M|\leq 2\cdot 2^{k-1}=2^k$.

(iii) In the case that $M$ has a row whose multiset of nonzero entries is $\{-1,-1\}$, $\{-1,1\}$ or $\{1,1\}$, the cofactor expansion of the row induces $|\det M|\leq 2^{k-1}+2^{k-1}=2^k$.

(iv) In the case that $M$ has a row $\boldsymbol{v}_i$ whose multiset of nonzero entries is $\{-2,1\}$, $\{-1,2\}$, $\{-2,1,1\}$ or $\{-1,-1,2\}$, assume that the $(i,j)$-entry is $-2$ or $2$ for some $j$. Then we can see the determinant as follows. 
\begin{eqnarray*}
    |\det M|  &=& \left|\det \left(\boldsymbol{w}_1,\dots,\boldsymbol{w}_{j-1},-\Sigma_{\ell=1}^k\boldsymbol{w}_\ell ,\boldsymbol{w}_{j+1},\dots,\boldsymbol{w}_k\right)\right|\\
    &=& \left|\det \begin{pmatrix}
   \pm1 & \pm1 & 0 & \cdots & 0\\
   \boldsymbol{w}'_1 & \boldsymbol{w}'_2 & \boldsymbol{w}'_3 & \cdots & \boldsymbol{w}'_k
   \end{pmatrix}\right|.
\end{eqnarray*}
Thus, this case comes down to the case (iii), and hence, we have $|\det M|\leq 2^k$.

(v) In the case that $M$ has a row $\boldsymbol{v}_i$ whose multiset of nonzero entries is $\{-2,2\}$, assume that the $(i,j)$-entry is $-2$ or $2$ for some $j$. Then we can see the determinant as follows.
\begin{eqnarray*}
    |\det M | &=& \left|\det\left(\boldsymbol{w}_1,\dots,\boldsymbol{w}_{j-1},-\Sigma_{\ell=1}^k\boldsymbol{w}_\ell ,\boldsymbol{w}_{j+1},\dots,\boldsymbol{w}_k\right)\right|\\
    &=& \left|\det\begin{pmatrix}
   \pm2 &  0 & \cdots & 0\\
   \boldsymbol{w}'_1 & \boldsymbol{w}'_2  & \cdots & \boldsymbol{w}'_k
   \end{pmatrix}\right|.
\end{eqnarray*}
Thus, this case comes down to the case (ii), and hence, we have $|\det M|\leq 2^k$.

(vi) In the case that $M$ has a row $\boldsymbol{v}_i$ whose multiset of nonzero entries is $\{-1,-1,1\}$ or $\{-1,1,1\}$, 
we can see the determinant as follows.
\begin{eqnarray*}
|\det M| &=&
\left|\det\begin{pmatrix}
   \pm1 & 1 & -1 & 0 & \cdots & 0\\
   \boldsymbol{w}'_1 & \boldsymbol{w}'_2 & \boldsymbol{w}'_3 & \boldsymbol{w}'_4 & \cdots & \boldsymbol{w}'_k
\end{pmatrix}\right|\\
 &=&
\left|\det\begin{pmatrix}
   \pm1 & 1 & 0 & 0 & \cdots & 0\\
   \boldsymbol{w}'_1 & \boldsymbol{w}'_2 & \boldsymbol{w}'_2+\boldsymbol{w}'_3 & \boldsymbol{w}'_4 & \cdots & \boldsymbol{w}'_k
\end{pmatrix}\right|\\ 
&\leq&
\left|\det\begin{pmatrix}
   \boldsymbol{w}'_2 & \boldsymbol{w}'_2+\boldsymbol{w}'_3 & \boldsymbol{w}'_4 & \cdots & \boldsymbol{w}'_k
\end{pmatrix}\right|\\
&&
+
\left|\det\begin{pmatrix}
   \boldsymbol{w}'_1 & \boldsymbol{w}'_2+\boldsymbol{w}'_3 & \boldsymbol{w}'_4 & \cdots & \boldsymbol{w}'_k
\end{pmatrix}\right|\\
&=&
\left|\det\begin{pmatrix}
   \boldsymbol{w}'_2 & \boldsymbol{w}'_3 & \boldsymbol{w}'_4 & \cdots & \boldsymbol{w}'_k
\end{pmatrix}\right|\\
&&+
\left|\det\begin{pmatrix}
   \boldsymbol{w}'_1 & -\Sigma_{\ell=1}^k{\boldsymbol{w}'_\ell} & \boldsymbol{w}'_4 & \cdots & \boldsymbol{w}'_k 
\end{pmatrix}\right|
\end{eqnarray*}
Since the matrices 
$\begin{pmatrix}
   \boldsymbol{w}'_2 & \cdots & \boldsymbol{w}'_k
\end{pmatrix}$
and
$\begin{pmatrix}
   \boldsymbol{w}'_1 & -\Sigma_{\ell=1}^k{\boldsymbol{w}'_\ell} & \boldsymbol{w}'_4 & \cdots & \boldsymbol{w}'_k 
\end{pmatrix}$
satisfy the condition $(\star)$,  
we obtain $|\det M|\leq 2^{k-1}+2^{k-1}=2^k$.

(vii) Lastly we consider the case that $M$ has a row $\boldsymbol{v}_i$ whose multiset of nonzero entries is $\{-1,-1,1,1\}$. 
We then have 
\begin{eqnarray*}
|\det M| &=&
\left|\det\begin{pmatrix}
   1 & 1 & -1 & -1 & 0 & \cdots & 0\\
   \boldsymbol{w}'_1 & \boldsymbol{w}'_2 & \boldsymbol{w}'_3 & \boldsymbol{w}'_4 & \boldsymbol{w}'_5 & \cdots & \boldsymbol{w}'_k
\end{pmatrix}\right|\\
 &=&
\left|\det\begin{pmatrix}
   1 & 1 & -1 & 0 & 0 & \cdots & 0\\
   \boldsymbol{w}'_1 & \boldsymbol{w}'_2 & \boldsymbol{w}'_3 & -\Sigma_{\ell=1}^k\boldsymbol{w}'_\ell & \boldsymbol{w}'_5 & \cdots & \boldsymbol{w}'_k
\end{pmatrix}\right|.
\end{eqnarray*}
This case comes down to the case (vi), and hence, we have $|\det M|\leq 2^k$.

This completes the proof.
\end{proof}

Let $p$ be an odd prime integer.
Let $D$ be a diagram of a knot $K$.
Let $x_1,  \ldots , x_{n+2}$ be the regions, and $c_1, \ldots , c_n$ the crossings of $D$. 
Let $C^{\rm 1T}$ and $C^{\rm 2T}$ be the $1$- and $2$-trivial Dehn $p$-coloring, respectively, such that $C^{\rm 1T}(x_1)=1$, $C^{\rm 2T}(x_1)=0$ and ${\rm Im}C^{\rm 2T}=\{0,1\}$.
Let $C$ be a nontrivial Dehn $p$-coloring of $D$.

\begin{lemma}\label{lem:case1}
Suppose that for any $i,j \in \{1, \ldots , n+2\}$ such that $C(x_i) =C(x_j)$, $C^{\rm 2T}(x_i) =C^{\rm 2T}(x_j)$ holds.
Then we have 
\[
\# \mathcal{C}(D,C) \geq \log_2 p +1.
\]
\end{lemma}
\begin{proof}
Put $\ell = \# \mathcal{C}(D,C)$ and $M_1= A_D(-1)$, where $A_D(-1)$ is the matrix constructed in Section~\ref{sec:coloringmatrices} by setting an arbitrary orientation for $D$.
By Lemma~\ref{lemma:colorequiv} and Remark~\ref{rem:colorequiv}, we may assume that $C(x_1)=0$ by replacing the Dehn $p$-coloring $C$ by a regular affine transformation if necessary. 
Put $\boldsymbol{a}_1 =\,^t(C(x_1) \cdots C(x_{n+2}))$
 and $\boldsymbol{a}_1^{\rm 2T} =\,^t(C^{\rm 2T}(x_1) \cdots C^{\rm 2T}(x_{n+2}))$.
Then by (2) of Lemma~\ref{lem:rank}, 
$$\rank_{\mathbb Z} M_1 =n+1 \mbox{ and }  \rank_{p} M_1 \leq n,$$
and the simultaneous equations $M_1 \boldsymbol{x} = \boldsymbol{0}$ have the linearly independent solution $\boldsymbol{a}_1^{\rm 2T}$ on $\mathbb Z$,  and two linearly independent solutions $\boldsymbol{a}_1^{\rm 2T}$ and $\boldsymbol{a}_1$ on $\mathbb Z_p$.

We put $M_1=(\boldsymbol{m}_1 \cdots  \boldsymbol{m}_{n+2} )$ and $\boldsymbol{x}=\,^t(x_1\cdots x_{n+2})$.
We transform $M_1$, $\boldsymbol{x}$, $\boldsymbol{a}_1$ and $\boldsymbol{a}_1^{\rm  2T}$ as follows. 
Suppose that $i,j\in\{1,\dots,n+2\}$ satisfy $i<j$ and $C(x_i)=C(x_j)$.
For $M_1$, we add $\boldsymbol{m}_j$ to $\boldsymbol{m}_i$ and delete $\boldsymbol{m}_j$.
For $\boldsymbol{x}$, $\boldsymbol{a}_1$ and $\boldsymbol{a}_1^{\rm 2T}$, we delete $x_j$, $C(x_j)$ and $C^{\rm 2T}(x_j)$, respectively.
We then have the new equation
$$(\boldsymbol{m}_1 \cdots \boldsymbol{m}_i+\boldsymbol{m}_j \cdots \widehat{\boldsymbol{m}_j} \cdots \boldsymbol{m}_{n+2}) ^t(x_1 \cdots x_i \cdots \widehat{x_j} \cdots x_{n+2})= \boldsymbol{0},$$
and the solutions
$$^t(C(x_1) \cdots C(x_i) \cdots \widehat{C(x_j)} \cdots C(x_{n+2}))$$
on $\mathbb{Z}_p$, and 
$$ ^t(C^{\rm 2T}(x_1) \cdots C^{\rm 2T}(x_i) \cdots \widehat{C^{\rm 2T}(x_j)} \cdots C^{\rm 2T}(x_{n+2}))$$
on  $\mathbb{Z}$ and $\mathbb{Z}_p$, where $\widehat{x}$ means that the entry $x$ is removed.
For the resultant matrix obtained from $M_1$ and the resultant column vectors obtained from $\boldsymbol{x}$, $\boldsymbol{a}_1$ and $\boldsymbol{a}_1^{\rm 2T}$, we repeat the same transformations unless the number of column vectors for the resultant matrix originally from $M_1$ is $\ell$.
Then, as a consequence, we have the $n\times\ell$-matrix $M_2$, the column vectors $\boldsymbol{y}=\ ^t(y_1,\dots,y_\ell)$, $\boldsymbol{a}_2=\ ^t(C(y_1),\dots,C(y_\ell))$ and $\boldsymbol{a}_2^{\rm 2T}=\ ^t(C^{\rm 2T}(y_1),\dots,C^{\rm 2T}(y_\ell))$, where $\boldsymbol{y}$ is the column vector originally from $\boldsymbol{x}$.

Remark that $\ell=\# \mathcal{C}(D,C)$. 
We can easily see that $\rank_\mathbb{Z}M_2=\ell-1$ and $\rank_pM_2\leq\ell-2$, and $\boldsymbol{a}^{\rm 2T}_2$ is a linearly independent solution of $M_2\boldsymbol{y}=\boldsymbol{0}$ on $\mathbb{Z}$, 
and $\boldsymbol{a}_2$ and $\boldsymbol{a}^{\rm 2T}_2$ are linearly independent solutions on $\mathbb{Z}_p$.
Hence, there exists an $(\ell-1)\times(\ell-1)$-submatrix $M_3$ of $M_2$ such that 
$\det M_3\neq 0\in\mathbb{Z}$ and $\det M_3=0\in\mathbb{Z}_p$, which implies that $p\leq |\det M_3|$. 

Furthermore, since we have $|\det M_3| \leq 2^{\ell-1}$ by Lemma \ref{lem:2^k}, it holds that $p\leq2^{\ell-1}$. This gives the conclusion $\#\mathcal{C}(D,C)=\ell \geq\log_2{p}+1$.
\end{proof}

\begin{lemma}\label{lem:case2}
Suppose that there exists $i,j \in \{1, \ldots , n+2\}$ with $i<j$ such that $C(x_i) =C(x_j)$ and $C^{\rm 2T}(x_i) \not =C^{\rm 2T}(x_j)$.
Then we have 
\[
\# \mathcal{C}(D,C) \geq \log_2 p +1.
\]
\end{lemma}

\begin{proof}
We prove this lemma in much the same way as Lemma \ref{lem:case1}.
Put $\ell = \# \mathcal{C}(D,C)$ and $M_1= B_{D,C; i, j}(-1)$, where $B_{D,C; i, j}(-1)$ is the matrix constructed in Section~\ref{sec:coloringmatrices} by setting an arbitrary orientation for $D$.
Put $\boldsymbol{a}_1 =\,^t(C(x_1) \cdots C(x_{n+2}))$
 and $\boldsymbol{a}_1^{\rm 1T} =\,^t(C^{\rm 1T}(x_1) \cdots C^{\rm 1T}(x_{n+2}))=\,^t(1,\dots,1)$.
Then by (3) of Lemma~\ref{lem:rank}, 
$$\rank_{\mathbb Z} M_1 =n+1 \mbox{ and }  \rank_{p} M_1 \leq n,$$
and the simultaneous equations $M_1 \boldsymbol{x} = \boldsymbol{0}$ have the linearly independent solution $\boldsymbol{a}_1^{\rm 1T}$ on $\mathbb Z$,  and two linearly independent solutions $\boldsymbol{a}_1^{\rm 1T}$ and $\boldsymbol{a}_1$ on $\mathbb Z_p$.

We put $M_1=(\boldsymbol{m}_1 \cdots  \boldsymbol{m}_{n+2} )$ and $\boldsymbol{x}=\,^t(x_1\cdots x_{n+2})$.
We transform $M_1$, $\boldsymbol{x}$, $\boldsymbol{a}_1$ and $\boldsymbol{a}_1^{\rm  1T}$ as follows. 
Suppose that $i,j\in\{1,\dots,n+2\}$ satisfy $i<j$ and $C(x_i)=C(x_j)$.
For $M_1$, we add $\boldsymbol{m}_j$ to $\boldsymbol{m}_i$ and delete $\boldsymbol{m}_j$.
For $\boldsymbol{x}$, $\boldsymbol{a}_1$ and $\boldsymbol{a}_1^{\rm 1T}$, we delete $x_j$, $C(x_j)$ and $C^{\rm 1T}(x_j)$, respectively.
We then have the new equation
$$(\boldsymbol{m}_1 \cdots \boldsymbol{m}_i+\boldsymbol{m}_j \cdots \widehat{\boldsymbol{m}_j} \cdots \boldsymbol{m}_{n+2}) ^t(x_1 \cdots x_i \cdots \widehat{x_j} \cdots x_{n+2})= \boldsymbol{0},$$
and the solutions
$$^t(C(x_1) \cdots C(x_i) \cdots \widehat{C(x_j)} \cdots C(x_{n+2}))$$
on  $\mathbb{Z}_p$, and 
$$ ^t(C^{\rm 1T}(x_1) \cdots C^{\rm 1T}(x_i) \cdots \widehat{C^{\rm 1T}(x_j)} \cdots C^{\rm 1T}(x_{n+2}))$$
on $\mathbb{Z}$ and $\mathbb{Z}_p$.
For the resultant matrix obtained from $M_1$ and the resultant column vectors obtained from $\boldsymbol{x}$, $\boldsymbol{a}_1$ and $\boldsymbol{a}_1^{\rm 1T}$, we repeat the same transformations unless the number of column vectors for the resultant matrix originally from $M_1$ is $\ell$.
Then, as a consequence, we have the $n\times\ell$-matrix $M_2$, the column vectors $\boldsymbol{y}=\ ^t(y_1,\dots,y_\ell)$, $\boldsymbol{a}_2=\ ^t(C(y_1),\dots,C(y_\ell))$ and $\boldsymbol{a}_2^{\rm 1T}=\ ^t(C^{\rm 1T}(y_1),\dots,C^{\rm 1T}(y_\ell))$, where $\boldsymbol{y}$ is the column vector originally from $\boldsymbol{x}$.

We can easily see that $\rank_{\mathbb Z} M_2 =\ell-1 \mbox{ and }  \rank_{p} M_2 \leq \ell-2$, and $\boldsymbol{a}_2^{\rm 1T}$ is a linearly independent solution of $M_2 \boldsymbol{y} = \boldsymbol{0}$ on $\mathbb Z$, and 
$\boldsymbol{a}_2$ and $\boldsymbol{a}_2^{\rm 1T}$ are linearly independent solutions on $\mathbb Z_p$.

Hence, there exists an $(\ell-1)\times(\ell-1)$-submatrix $M_3$ of $M_2$ such that 
$\det M_3\neq 0\in\mathbb{Z}$ and $\det M_3=0\in\mathbb{Z}_p$, which implies that $p\leq |\det M_3|$. 
Furthermore, since we have $|\det M_3| \leq 2^{\ell-1}$ by Lemma \ref{lem:2^k}, it holds that $p\leq2^{\ell-1}$. This gives the conclusion $\#\mathcal{C}(D,C)=\ell \geq\log_2{p}+1$.

\end{proof}

\begin{proof}[Proof of Theorem~\ref{th:main1}.]
From Lemmas \ref{lem:case1} and \ref{lem:case2}, 
we have \[
\# \mathcal{C}(D,C) \geq \log_2 p +1
\]
for any nontrivially Dehn $p$-colored diagram $(D, C)$ of a knot. 
Therefore, we have 
\[
\mincol_p(K) \geq \log_2 p +1
\]
for any Dehn $p$-colorable knot $K$.
\end{proof}

\section{$\mathcal{R}$-palette graphs}\label{sec:palettegraph}

Let $p$ be an odd prime number, and let $S \subset \mathbb Z_p$ with $S\not = \emptyset$. 
Set $\mu(S)= \{\{a_1,a_2\} ~|~ a_1, a_2 \in S\}$, where $\{a_1,a_2\}$ is regarded as the multiset $\{a_1, a_1\}$ when $a_1=a_2$.

For $\{a_1,a_2\}, \{a_3,a_4\} \in \mu(S)$, we set 
$$\{a_1,a_2\} \sim \{a_3,a_4\}\mbox{ if $a_1+a_2 = a_3+a_4$ in $\mathbb Z_p$}.$$
We note that this relation $\sim$ on $\mu(S)$ is an equivalence relation. We denote by $\overline{a_1+a_2}$ the equivalence class of $\{a_1,a_2\}\in \mu(S)$. 

\begin{definition}\label{def:R-palettegraph}
The {\it $\R$-palette graph of $S$} is the simple graph $G_{S}=(V_{S}, E_{S})$ composed of the vertex set 
\[
V_{S}= \mu(S)/\sim = \{\oline{a_1+a_2} ~|~ \{a_1,a_2\} \in \mu(S) \}
\]
and the edge set $E_{S}$ satisfying that 
$$e= \oline{b_1} \, \oline{b_2} \in E_{S} \iff \begin{array}{l}
\mbox{there exist $\{a_1,a_2\} \in \oline{b_1}$ and $\{a_3,a_4\} \in \oline{b_2}$ such }\\ \mbox{that $a_1+a_3 = a_2+a_4$ or 
$a_1+a_4 = a_2+a_3$ in $\mathbb Z_p$,}
\end{array}
$$
where $e= \oline{b_1} \, \oline{b_2}$ means that $e$ is an edge connecting the vertices $\oline{b_1}$ and $\oline{b_2}$. 
We attach the label $\oline{2^{-1}(b_1+b_2)} \in \mathbb Z_p$ to the edge $e$ between $\oline{b_1}$ and $\oline{b_2}$ (see Figure~\ref{PaletteGraph7} for example).
\begin{figure}[ht]
  \begin{center}
    \includegraphics[clip,width=12cm]{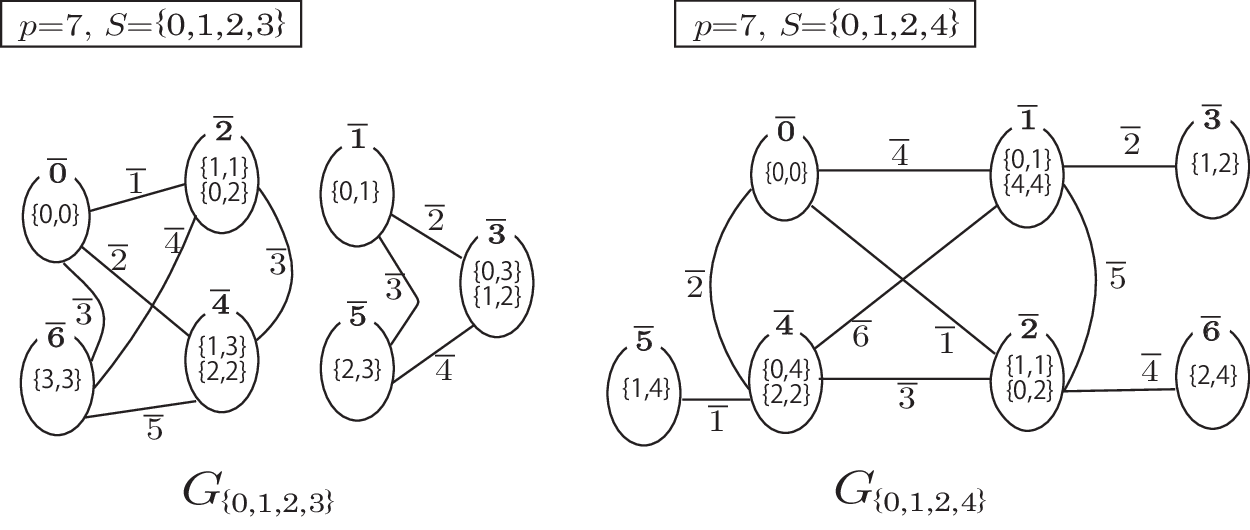}
    \caption{$\R$-palette graphs}
    \label{PaletteGraph7}
  \end{center}
\end{figure}

Let $G_S$ be the $\R$-palette graph of $S$. 
A graph $G=(V, E)$ is an {\it $\R$-subgraph} of $G_S$ if $G$ is a subgraph of $G_S$, and  
$$e\in E \Longrightarrow \oline{b_e} \in V$$
$$( \mbox{i.e., } e= \oline{b_1} \, \oline{b_2} \in E \Longrightarrow \oline{2^{-1}(b_1+b_2)} \in V)$$
holds, where $\oline{b_e}$ is the label of $e$ (see Figure~\ref{PaletteGraph7-2} for example).
\begin{figure}[ht]
  \begin{center}
    \includegraphics[clip,width=10cm]{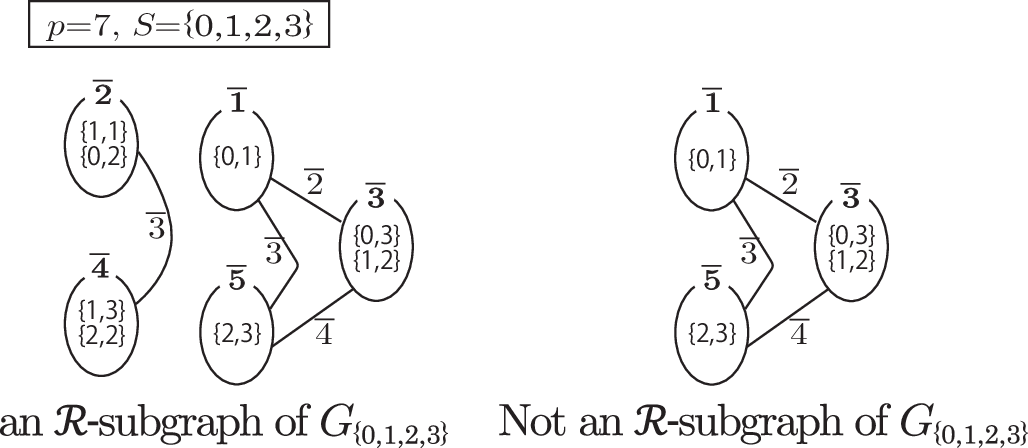}
    \caption{$\R$-palette graphs}
    \label{PaletteGraph7-2}
  \end{center}
\end{figure}

Let $(D,C)$ be a nontrivially Dehn $p$-colored diagram of a knot.
The {\it $\mathcal{R}$-palette graph of $(D,C)$} is an $\R$-subgraph $G_{(D,C)} = (V_{(D,C)}, E_{(D,C)})$ of $G_{\mathcal C(D,C)}$ composed of the vertex set 
\[
V_{(D,C)} = \{ \oline{b} ~|~ \mbox{there exists a $\overline{b}$-arc on $(D,C)$}  \}
\]
and the edge set $E_{(D,C)}$ satisfying that 
\[
e= \oline{b_1} \, \oline{b_2} \in E_{(D,C)} \iff \begin{minipage}{8cm}
 there exists a crossing on $(D,C)$ consisting of both the $\oline{b_1}$- and $\oline{b_2}$-under-arcs.
\end{minipage}
\]
As in the case of $G_S$, we attach the label $\oline{2^{-1}(b_1+b_2)} \in \mathbb Z_p$ to the edge $e$ between $\oline{b_1}$ and $\oline{b_2}$ (see Figure~\ref{PaletteGraph7-3} for example). 

We remark that $\bar{b}$ for a $\bar{b}$-arc was defined just a symbol, but $\bar{b}$ is regarded as an equivalence class as in Definition~\ref{def:R-palettegraph} when it represents a vertex. 

We also remark that this $\mathcal{R}$-palette graph of $(D,C)$ is isomorphic to the palette graph of $(D,\bar{C})$ defined in \cite{NakamuraNakanishiSatoh13}, where $\bar{C}$ is the Fox $p$-coloring corresponding to $C$ as shown in Remark~\ref{remark:1.1}.
Furthermore, it was shown in \cite{NakamuraNakanishiSatoh13} that the palette graph of $(D,\bar{C})$ is connected and has at least $\lfloor \log_2 p \rfloor +2$ vertices.
\begin{figure}[ht]
  \begin{center}
    \includegraphics[clip,width=7cm]{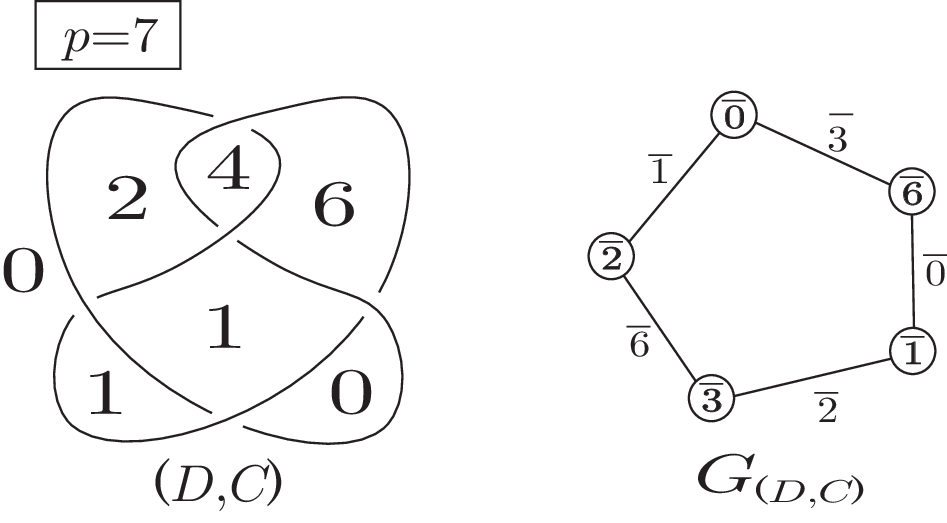}
    \caption{The $\R$-palette graph of $(D,C)$}
    \label{PaletteGraph7-3}
  \end{center}
\end{figure}
\end{definition}

\section{Evaluations of $\mincol_p(K)$ by $\R$-palette graphs}\label{evaluation of graph}

Let $p$ be an odd prime number.

The next theorem can be used to evaluate the minimum number of colors of knots or select some sets of colors, as candidates, each of which might be the set of colors for a nontrivial Dehn $p$-coloring of a knot diagram.  
\begin{theorem}\label{th:palette}
Let $S \subset \mathbb Z_p$.
If $S= \C(D,C)$ for some nontrivially Dehn $p$-colored diagram $(D,C)$ of a knot, the $\R$-palette graph $G_S$ includes a connected $\R$-subgraph with at least $\lfloor \log_2 p \rfloor +2$ vertices. 
\end{theorem}
\begin{proof}
This is evident because $G_S$ includes $G_{(D,C)}$ and as remarked in the previous section, $G_{(D,C)}$ is connected and has at least $\lfloor \log_2 p \rfloor +2$ vertices. 
\end{proof}

\begin{example}\label{ex:palettegraph}
For any subset $S\subset \mathbb Z_7$ such that $\#S \leq 3$, 
the $\R$-palette graph $G_S$ includes no connected $\R$-subgraph with at least $\lfloor \log_2 7 \rfloor +2$ vertices. This and Theorem~\ref{th:palette} imply that $\mincol_7(K) \geq 4$ for any Dehn $7$-colorable knot $K$.

Any subset of $\mathbb Z_7$ with four elements is affine equivalent to either $\{0,1,2,3\}$ or $\{0,1,2,4\}$. 
We note that by Lemma~\ref{lemma:colorequiv} and Remark~\ref{rem:colorequiv}, when we select candidates of sets of four colors each of which might be the set of colors for a nontrivially Dehn $p$-colored diagram of a knot, it suffices to discuss only the representatives of the affine equivalence classes, that is $\{0,1,2,3\}$ and $\{0,1,2,4\}$.
As shown in the right of Figure~\ref{PaletteGraph7}, the $\R$-palette graph $G_{\{0,1,2,4\}}$ clearly has a connected $\R$-subgraph with at least $\lfloor \log_2 7 \rfloor +2$ vertices, which is $G_{\{0,1,2,4\}}$ itself for example.
On the other hand, the $\R$-palette graph $G_{\{0,1,2,3\}}$ depicted in the left of Figure~\ref{PaletteGraph7} does not include a connected $\R$-subgraph with at least $\lfloor \log_2 7 \rfloor +2$ vertices.
Therefore, if there exists a nontrivially Dehn $7$-colored diagram $(D, C)$ for a knot with $\# \C(D,C)=4(= \lfloor \log_2 7 \rfloor +2)$, then $\C(D,C)$ is affine equivalent to $\{0,1,2,4\}$, and moreover, there exists a Dehn $7$-coloring $C'$ of $D$ with $C \sim C'$ and $\C(D,C')=\{0,1,2,4\}$.
\end{example}

Similarly as Example~\ref{ex:palettegraph}, for each odd prime $p$ with $p< 2^5$, we can give the candidates of $\lfloor \log_2 p \rfloor +2$ colors by using $\R$-palette graphs, where $\lfloor \log_2 p \rfloor +2$ is the lower bound of the evaluation formula of Theorem~\ref{th:main1}. The next theorem gives the results.
\begin{theorem}\label{prop:colorcandidate}
Let $p$ be an odd prime number with $p< 2^5$.
\begin{itemize}
\item[(1)] There is no subset $S\subset \mathbb Z_p$ with at most $\lfloor \log_2 p \rfloor +1$ elements such that the $\R$-palette graph $G_S$ includes a connected $\R$-subgraph with at least $\lfloor \log_2 p \rfloor +2$ vertices. This implies that 
\[
\mincol_p(K) \geq \lfloor \log_2 p \rfloor +2
\]
for any Dehn $p$-colorable knot $K$.

\item[(2)] If there exists a nontrivially Dehn $p$-colored diagram $(D,C)$ of a knot such that $\#\C(D,C)= \lfloor \log_2 p \rfloor +2$, then 
\begin{itemize}
\item[(i)] $\C(D,C) \sim \{0,1,2\}$ when $p=3$,  
\item[(ii)] $\C(D,C) \sim \{0,1,2,3\}$ when $p=5$, 
\item[(iii)] $\C(D,C) \sim \{0,1,2,4\}$ when $p=7$, 
\item[(iv)] $\C(D,C) \sim \{0,1,2,3,6\}$ or $ \{0,1,2,4,7\}$ when $p=11$, 
\item[(v)] $\C(D,C) \sim \{0,1,2,4,7\}$ when $p=13$, 
\item[(vi)] $\C(D,C) \sim \{0,1,2,3,5,9\}$, $\{0,1,2,3,5,10\}$, $\{0,1,2,3,5,12\}$, $\{0,1,2,3,6,9\}$, $\{0,1,2,3,6,10\}$, $\{0,1,2,3,6,11\}$, $\{0,1,2,3,6,13\}$, $\{0,1,2,3,7,11\}$, $\{0,1,2,4,5,9\}$, $\{0,1,2,4,5,10\}$,  $\{0,1,2,4,5,12\}$,  or $\{0,1,2,4,10,13\}$ when $p=17$, 
\item[(vii)] $\C(D,C) \sim \{0,1,2,3,5,10\}$, $\{0,1,2,3,6,10\}$, $\{0,1,2,3,6,11\}$, $\{0,1,2,3,6,12\}$, $\{0,1,2,3,6,13\}$, $\{0,1,2,3,6,14\}$, $\{0,1,2,3,7,12\}$, $\{0,1,2,4,5,10\}$, $\{0,1,2,4,5,14\}$, $\{0,1,2,4,7,12\}$, or $\{0,1,2,4,7,15\}$ when $p=19$, 
\item[(viii)] $\C(D,C) \sim \{0,1,2,3,6,12\}$, $\{0,1,2,4,7,12\}$, $\{0,1,2,4,7,13\}$, $\{0,1,2,4,7,14\}$, $\{0,1,2,4,9,14\}$, or $\{0,1,2,4,10,19\}$ when $p=23$, 
\item[(ix)] $\C(D,C) \sim \{0,1,2,4,8,15\}$ when $p=29$, and 
\item[(x)] $\C(D,C) \sim \{0,1,2,4,8,16\}$ when $p=31$.
\end{itemize}
\end{itemize}
\end{theorem}

\begin{proof}
This can be shown by Mathematica computations.
In particular, we found a connected $\R$-subgraph with at least $\lfloor \log_2 p \rfloor +2$ vertices for the $\R$-palette graph of each candidate of the sets of colors given in (2) of this theorem as depicted in Figures \ref{con1}--\ref{con4}.
\end{proof}

\begin{figure}[ht]
  \begin{center}
\includegraphics[clip,width=12cm]{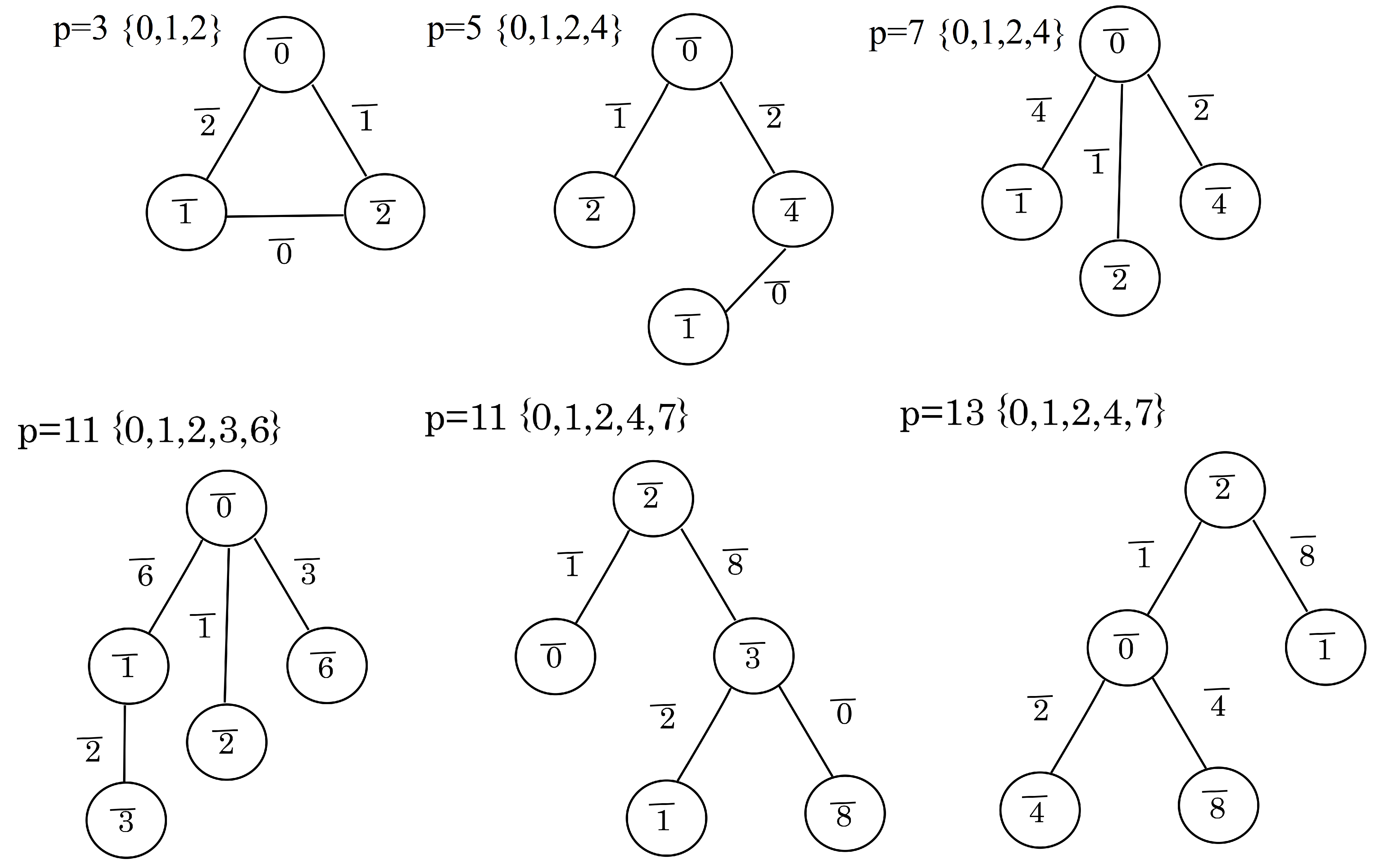}

    \caption{Connected $\R$-subgraphs for $p\in\{3,5,7,11,13\}$}
    \label{con1}
  \end{center}
\end{figure}

\begin{figure}[H]
  \begin{center}    \includegraphics[clip,width=13cm]{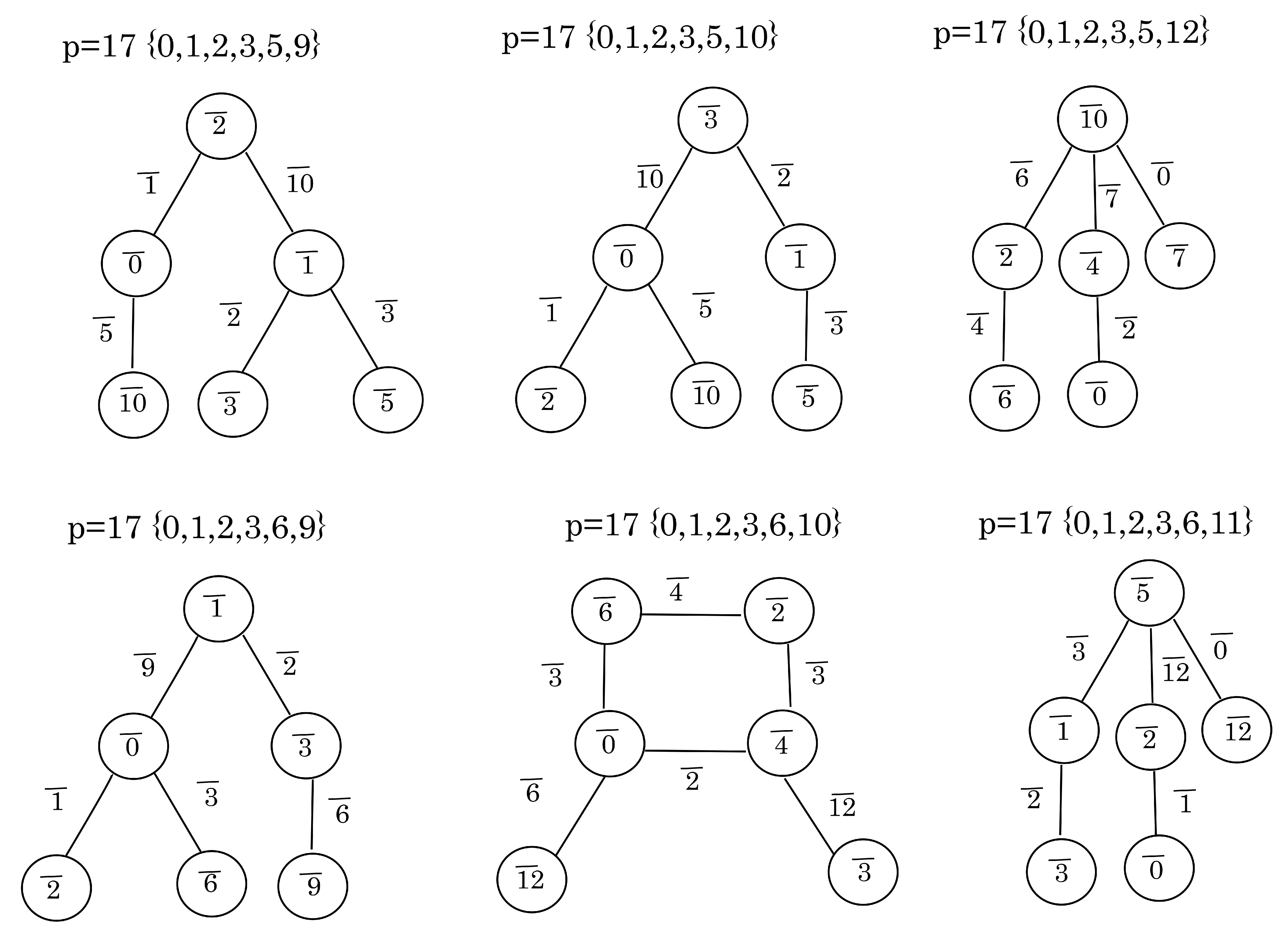}
    \caption{Connected $\R$-subgraphs for $p\in\{17\}$}
    \label{con2-1}
  \end{center}
\end{figure}

  \begin{figure}[H]
  \begin{center}  \includegraphics[clip,width=13cm]{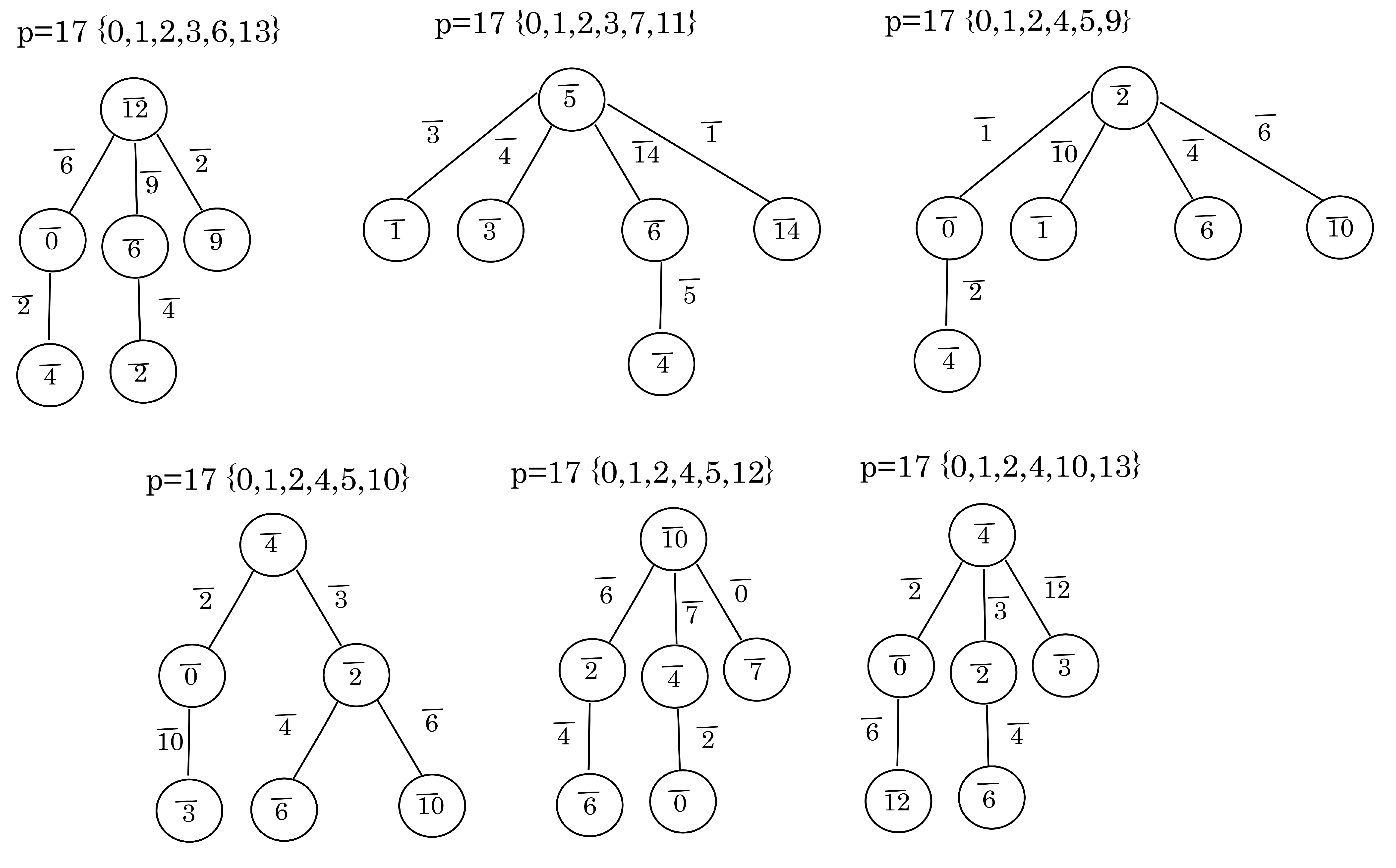}
    \caption{Connected $\R$-subgraphs for $p\in\{17\}$}
    \label{con2-2}
  \end{center}
\end{figure}

\begin{figure}[H]
  \begin{center}    \includegraphics[clip,width=13cm]{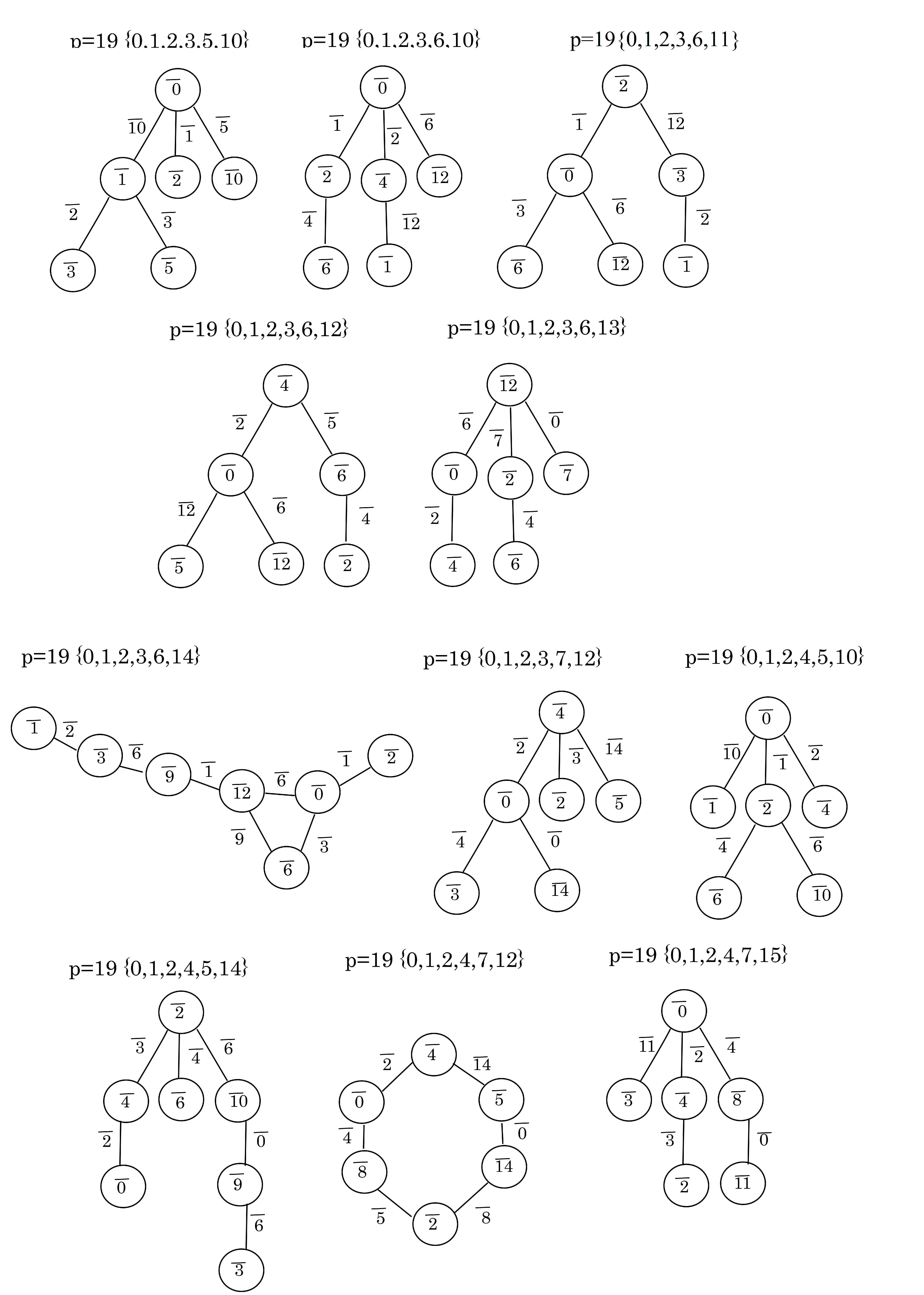}
 \caption{Connected $\R$-subgraphs for $p\in\{19\}$}
 \label{con3}
  \end{center}
\end{figure}

\begin{figure}[H]
  \begin{center}    \includegraphics[clip,width=12cm]{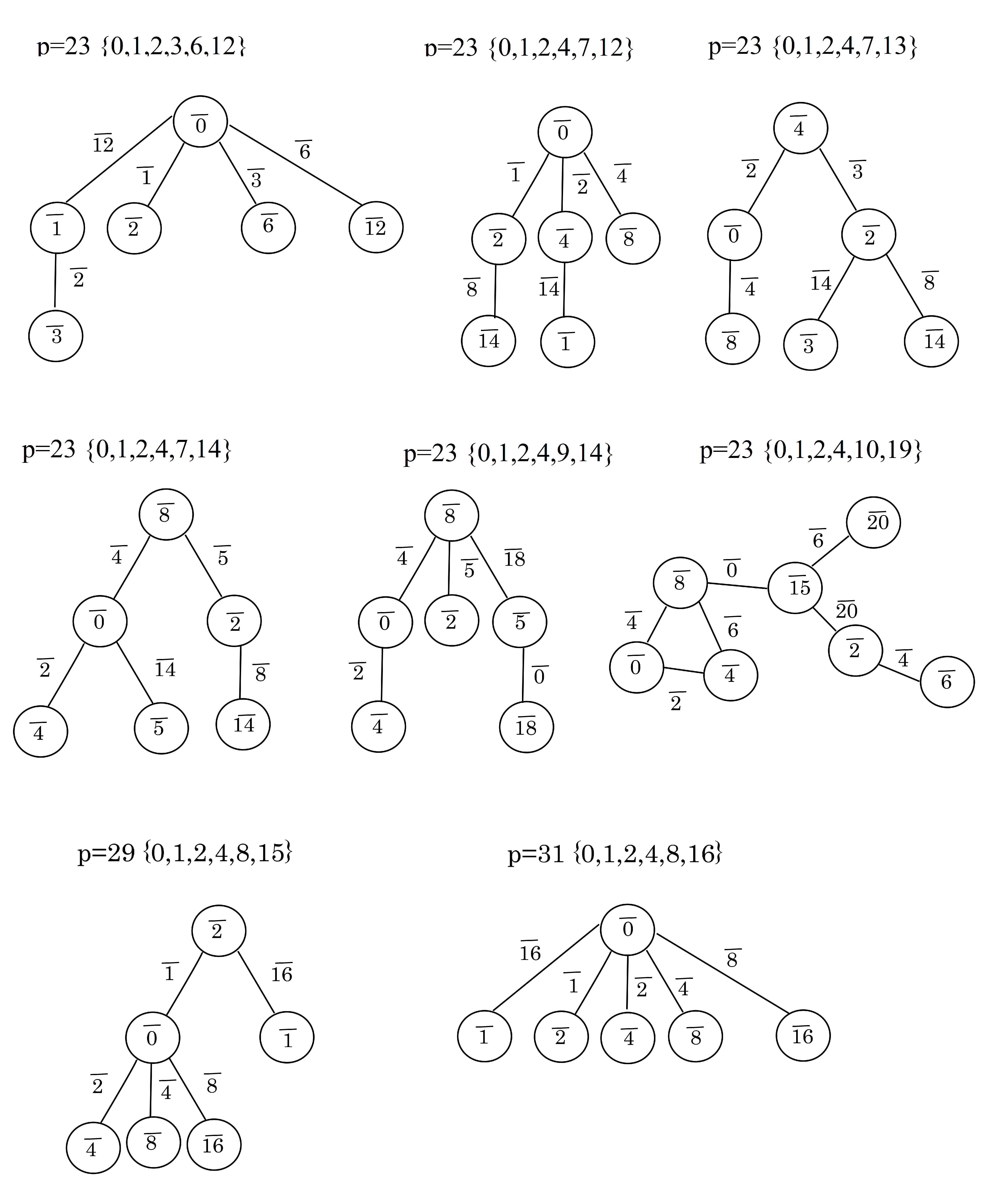}
    \caption{Connected $\R$-subgraphs for $p\in\{23,29,31\}$}
        \label{con4}
  \end{center}
\end{figure}

We have the following proposition.
\begin{proposition}\label{prop:pcoloreddiagram}
For each odd prime number $p$ with $p<2^5$ and $p\not \in \{13, 29\}$, there exists a Dehn $p$-colorable knot $K$ with $\mincol_p(K) =\lfloor \log_2 p \rfloor +2$. 
\end{proposition}

\begin{proof}
For $p= 3$ (resp. $5, 7,11,17,19,23,31$), the diagram illustrated in Figure~\ref{prop3} (resp.  \ref{prop5}, \ref{prop7}, \ref{prop11}, \ref{prop17}, \ref{prop19}, \ref{prop23}, \ref{prop31}) admits the Dehn $p$-coloring with $\lfloor \log_2 p \rfloor +2$ colors. Hence, from these properties and Theorem~\ref{th:main1}, we can see that there exists a knot $K$ with $\mincol_p(K) =\lfloor \log_2 p \rfloor +2$ for each $p\in \{3,5,7,11,17,19,23, 31\}$.

\begin{figure}[H]
\begin{minipage}[b]{0.45\linewidth}
    \centering    \includegraphics[clip,width=2cm]{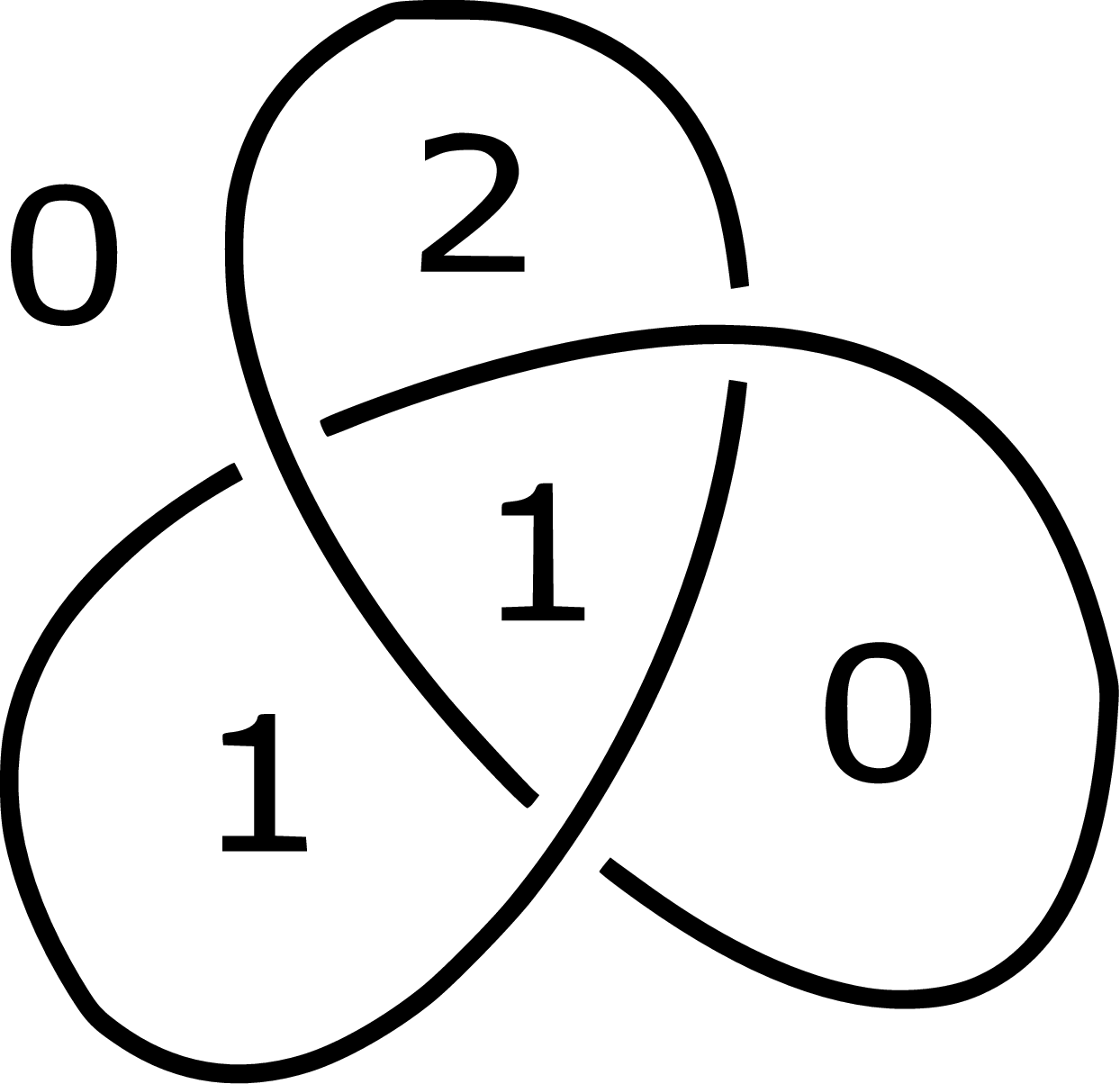}
    \caption{A Dehn $3$-colored diagram}
    \label{prop3}
  \end{minipage}
  \begin{minipage}[b]{0.45\linewidth}
    \centering
    \includegraphics[clip,width=2cm]{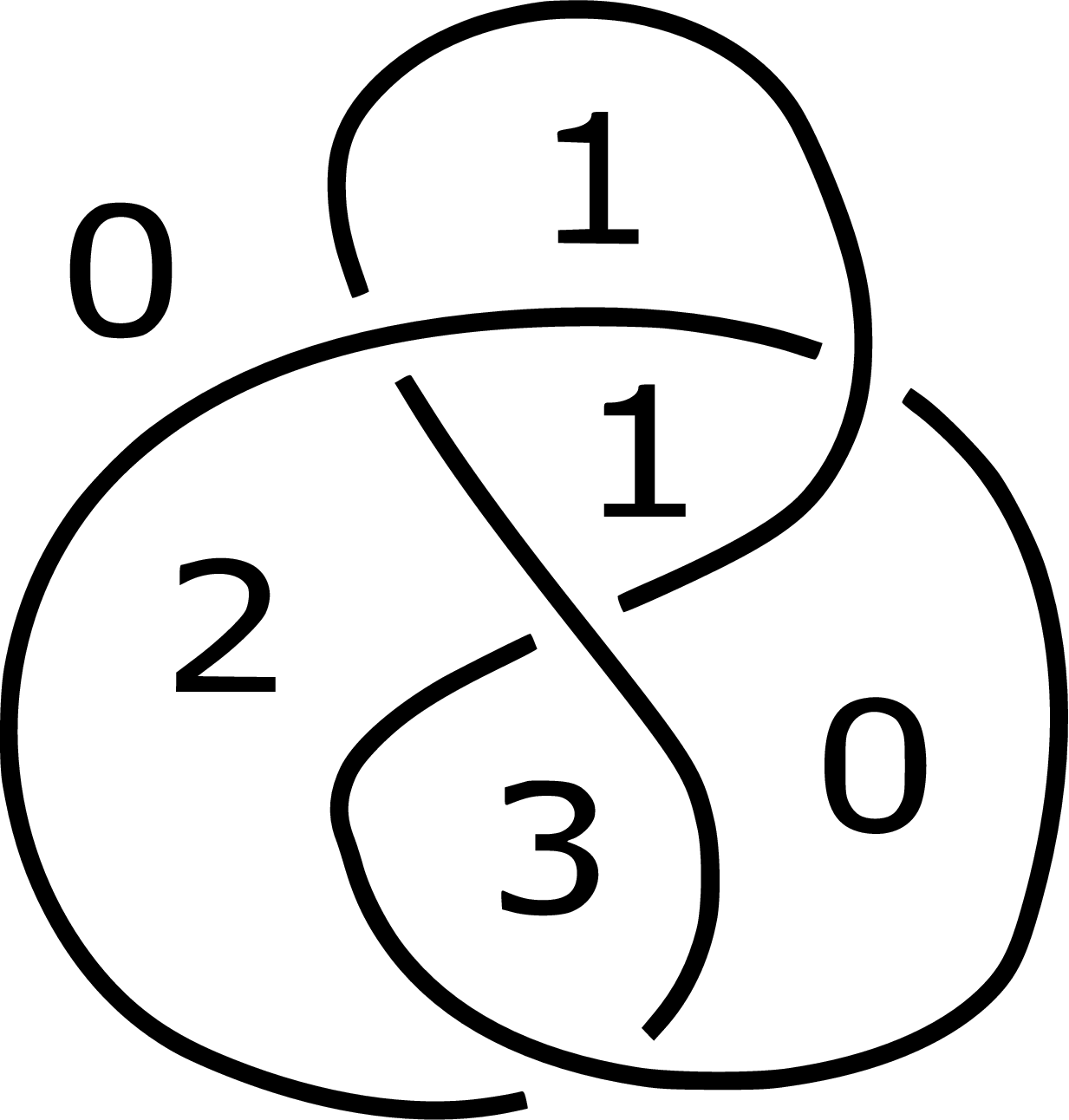}
    \caption{A Dehn $5$-colored diagram}
      \label{prop5}
  \end{minipage}
\end{figure}

\begin{figure}[H]
\begin{minipage}[b]{0.45\linewidth}
    \centering    \includegraphics[clip,width=4cm]{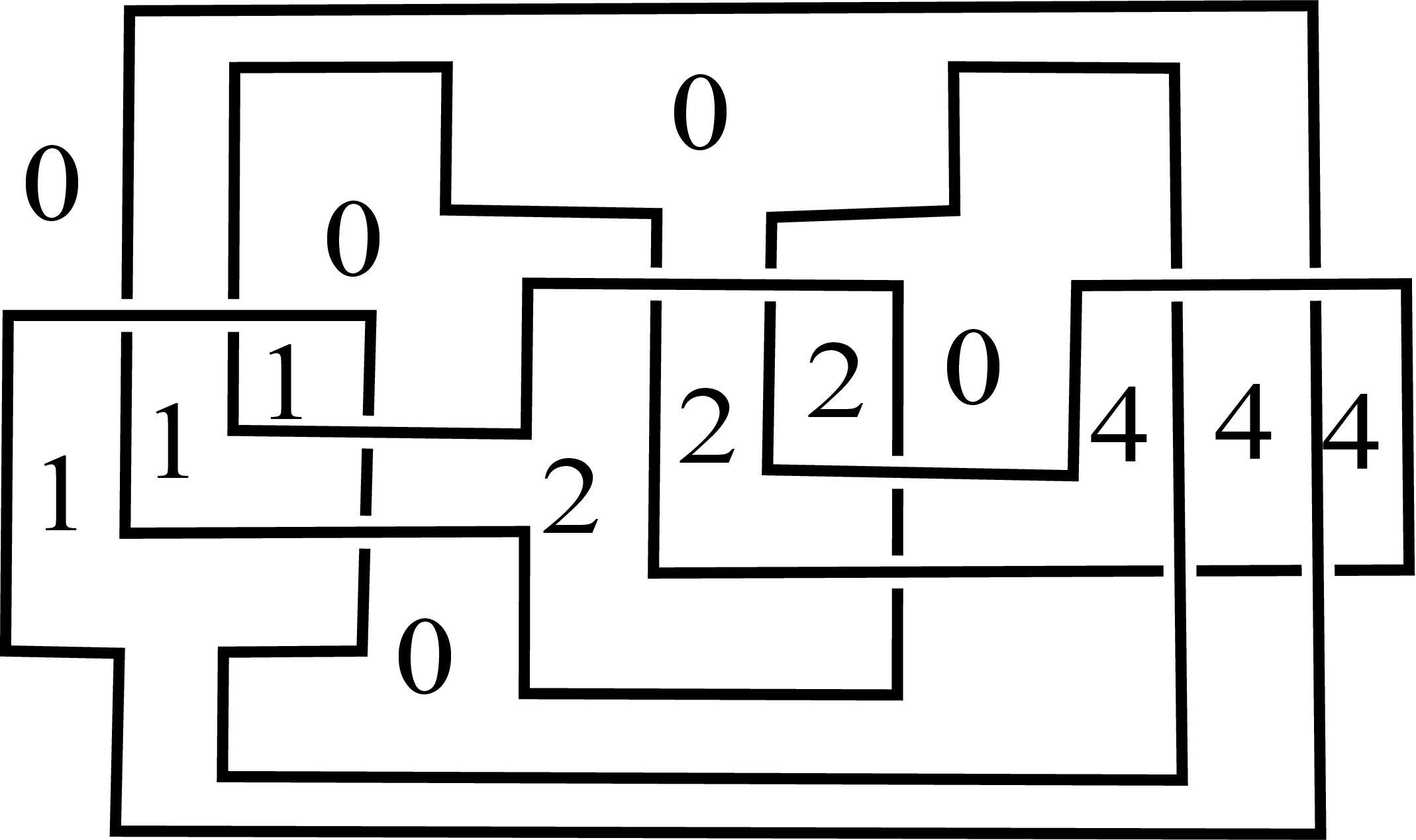}
    \caption{A Dehn $7$-colored diagram}
    \label{prop7}
  \end{minipage}
  \begin{minipage}[b]{0.45\linewidth}
    \centering
    \includegraphics[clip,width=4cm]{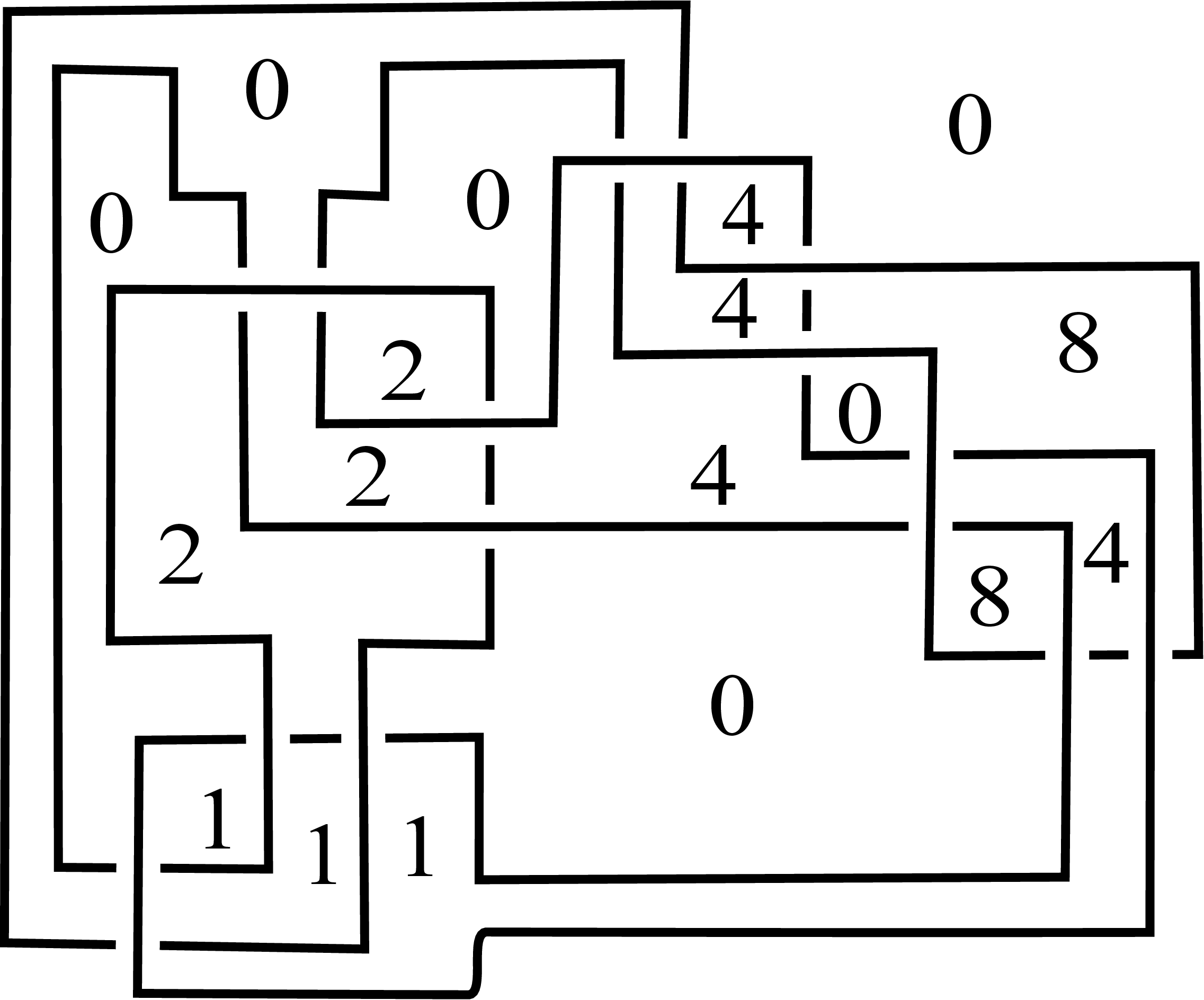}
    \caption{A Dehn $11$-colored diagram}
      \label{prop11}
  \end{minipage}
\end{figure}

\begin{figure}[H]
\begin{minipage}[b]{0.45\linewidth}
    \centering
    \includegraphics[clip,width=4cm]{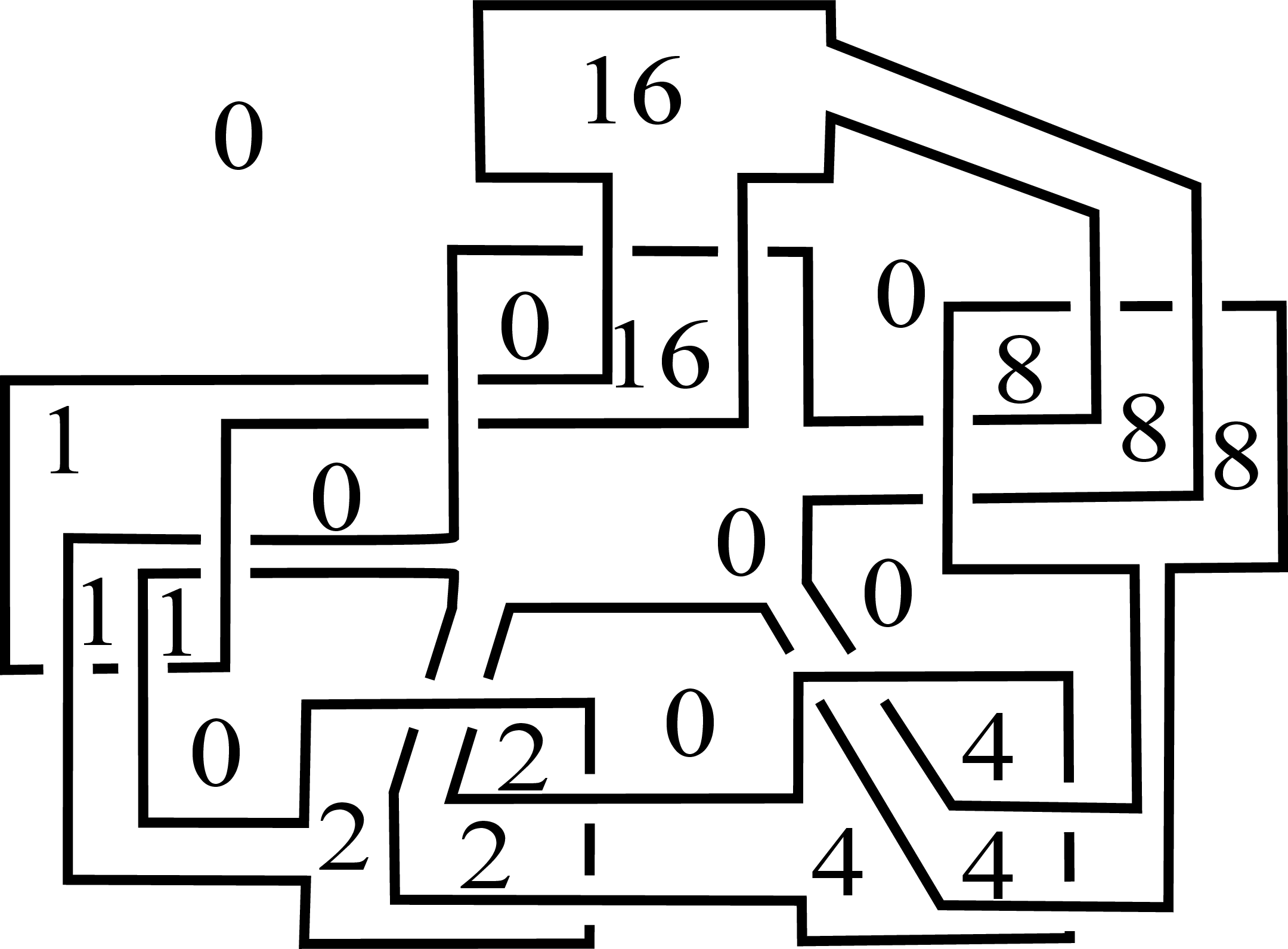}
    \caption{A Dehn $17$-colored diagram}
    \label{prop17}
  \end{minipage}
  \begin{minipage}[b]{0.45\linewidth}
    \centering
    \includegraphics[clip,width=4cm]{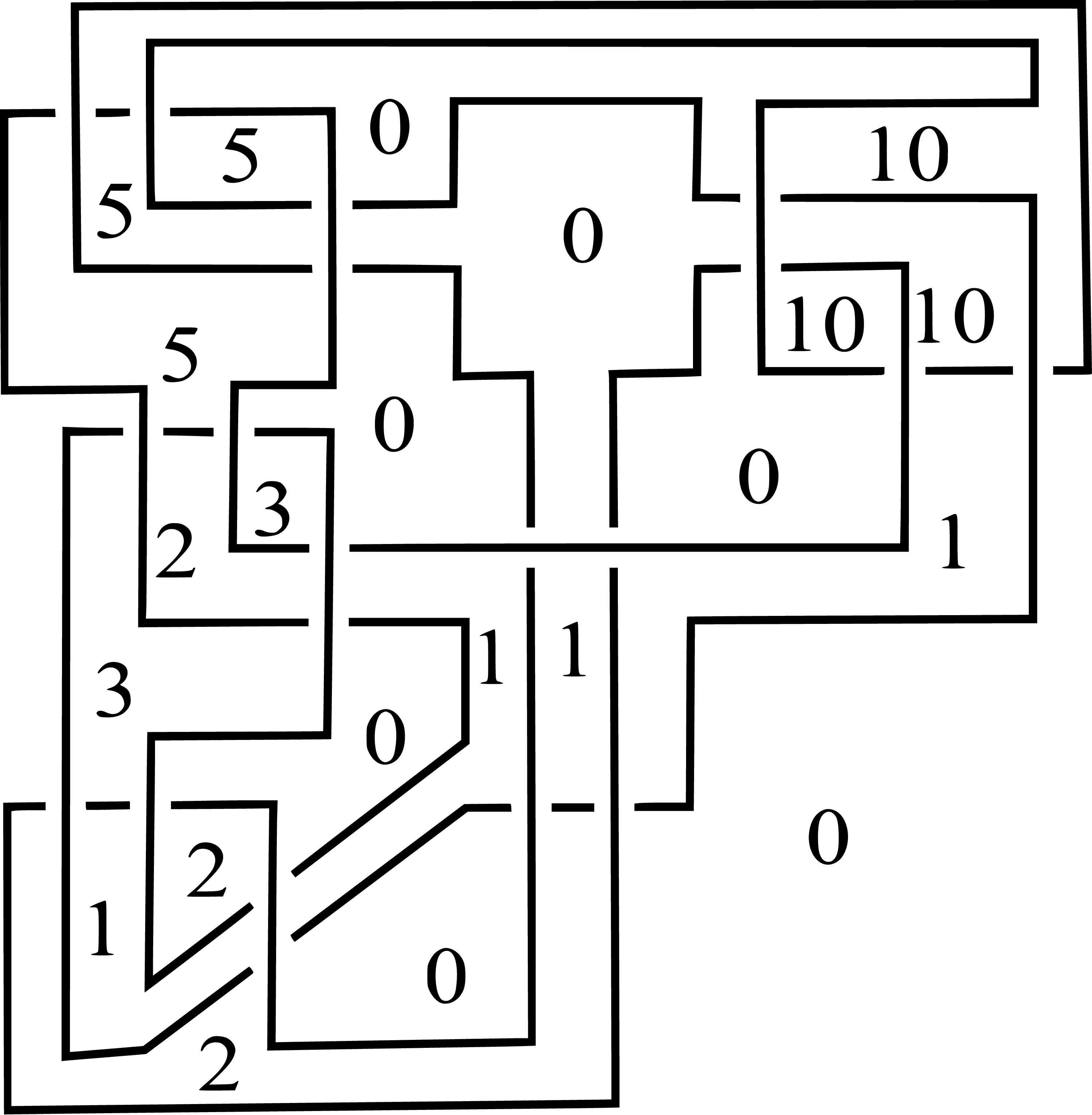}
    \caption{A Dehn $19$-colored diagram}
      \label{prop19}
  \end{minipage}
\end{figure}

\begin{figure}[H]
\begin{minipage}[b]{0.45\linewidth}
    \centering
    \includegraphics[clip,width=4cm]{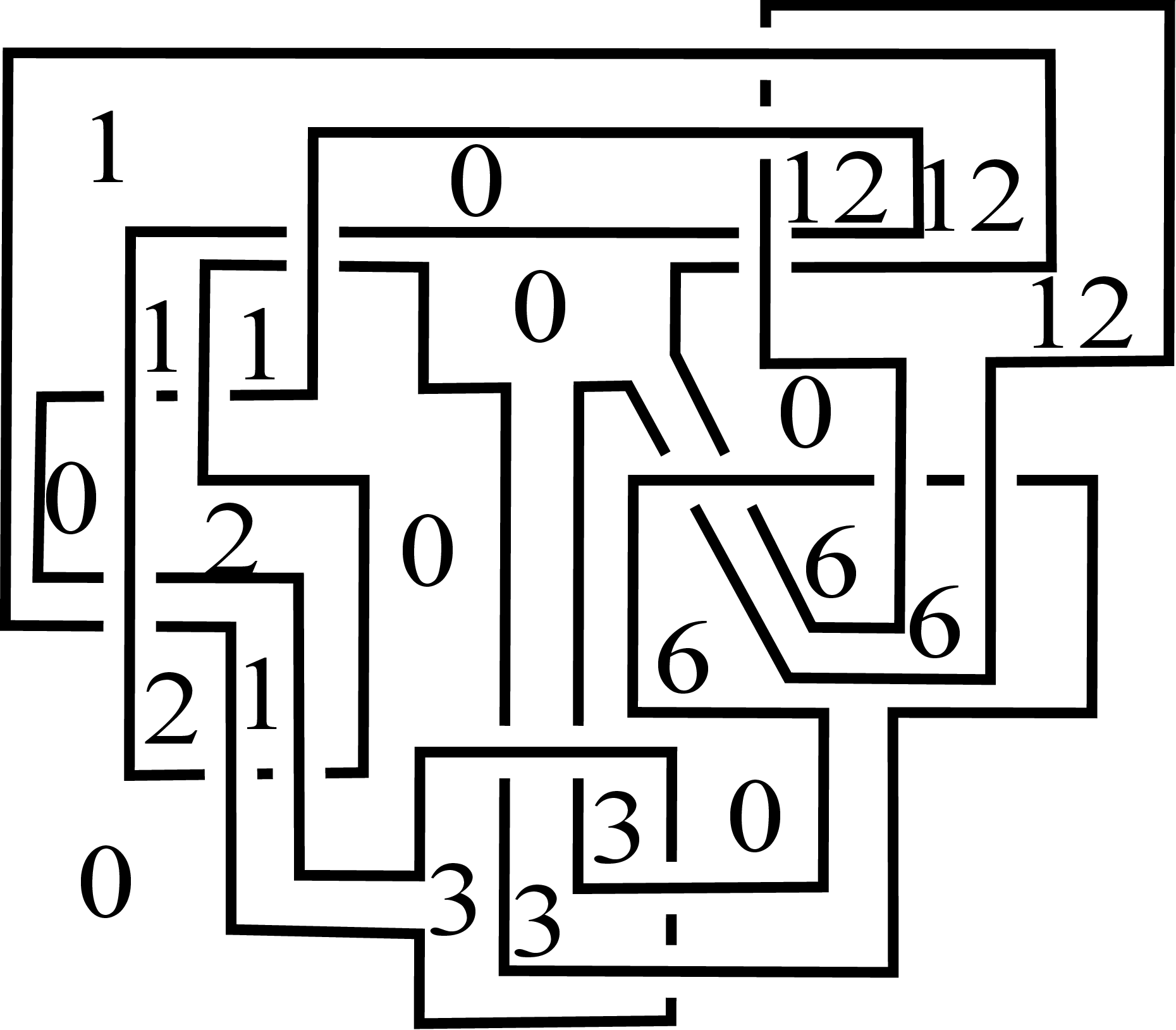}
    \caption{A Dehn $23$-colored diagram}
    \label{prop23}
  \end{minipage}
  \begin{minipage}[b]{0.45\linewidth}
    \centering
    \includegraphics[clip,width=4cm]{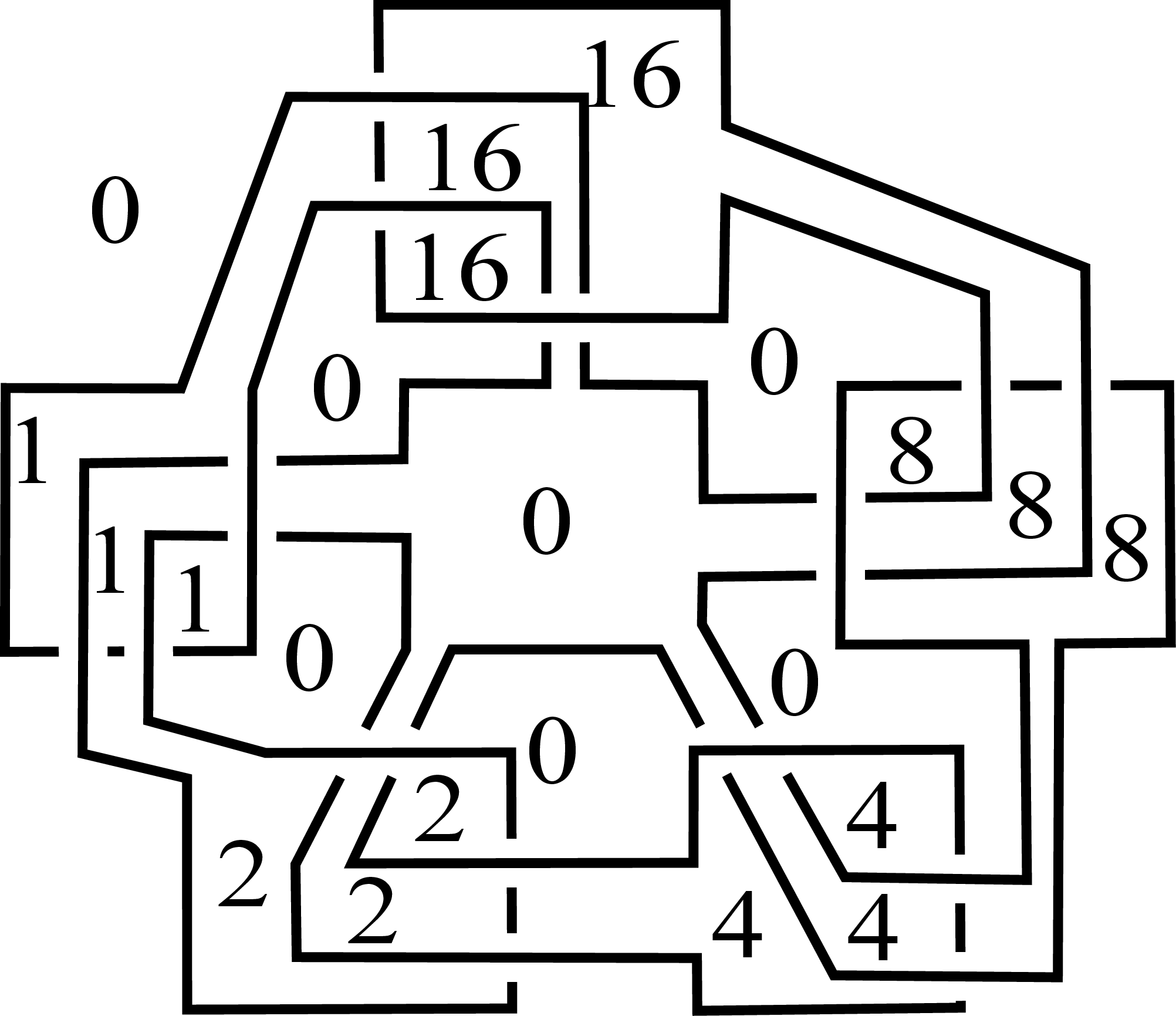}
    \caption{A Dehn $31$-colored diagram}
      \label{prop31}
  \end{minipage}
\end{figure}

\end{proof}

\begin{remark}
For $p=13$ or $29$, there is no Dehn $p$-colorable knot $K$ with $\mincol_p(K) =\lfloor \log_2 p \rfloor +2$. 
This result will be shown in our next paper \cite{MatsudoOshiroYamagishi-2}.
\end{remark}

\begin{remark}\label{rem}
The following properties hold.
\begin{itemize}
\item[(1)] When $p=3$, any $p$-colorable knot $K$ has $\mincol_p(K) =3$. 
\item[(2)] When $p=5$, any $p$-colorable knot $K$ has $\mincol_p(K) =4$. 
\end{itemize}
The proof of (1) is left as an exercise to the reader.
We give the proof of (2) in Appendix~\ref{appendix:5-colorableknot}.
\end{remark}

\appendix 
\section{} \label{appendix:5-colorableknot}

\begin{proof}[Proof of Remark \ref{rem}~(2)]
In this proof, an {\it $n$-gon} means a region bounded by $n$ semiarcs.

Let $(D,C)$ be any nontrivially Dehn $5$-colored diagram of a knot. 
We may construct a Dehn $5$-colored diagram $(D',C')$ with four colors by removing regions with the color $4$ from $(D,C)$ in the following steps. 
\begin{itemize}
    \item(Step 1) Remove crossings that have four or two regions colored by $4$.
    \item(Step 2) Remove $n$-gons
        colored by $4$ for $n\geq 4$.
    \item(Step 3) Remove $3$-, and $2$-gons colored by $4$.
\end{itemize}

In this proof, suppose that $a,b,c$ are  colors with $a\neq b$, $a\neq c$, $a,b,c\not =4$. 

(Step 1) 
First, we remove crossings that have only the color $4$ by the move depicted by Figure~\ref{app01}, where we note that there exists a region colored by some $a$.
\begin{figure}[H]
  \begin{center}    \includegraphics[clip,width=4cm]{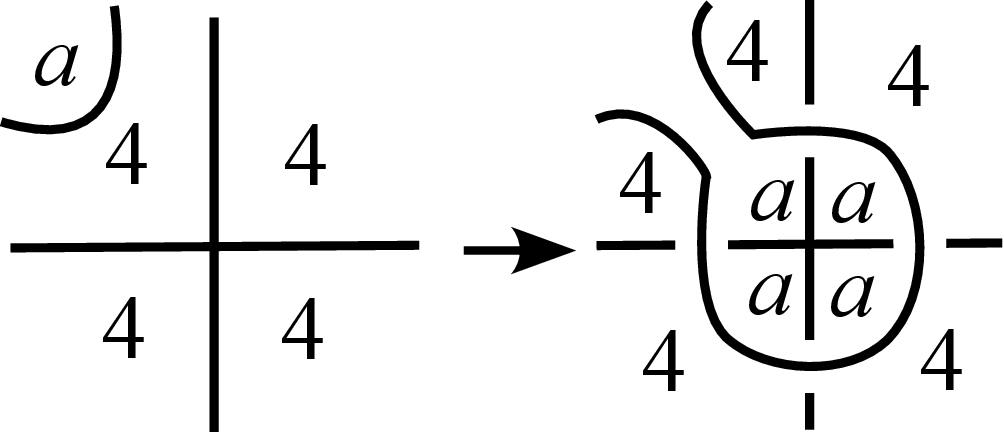}
    \caption{}
    \label{app01}
  \end{center}
\end{figure}

Next, we remove $\{4,4\}$-semiarcs by the moves as depicted in Figures \ref{app02}--\ref{app06}, where
the situation around each $\{4,4\}$-semiarc is represented by one of Figures \ref{app02}--\ref{app06}. 
\begin{figure}[H]
\begin{minipage}[b]{0.45\linewidth}
  \centering    \includegraphics[clip,height=2cm]{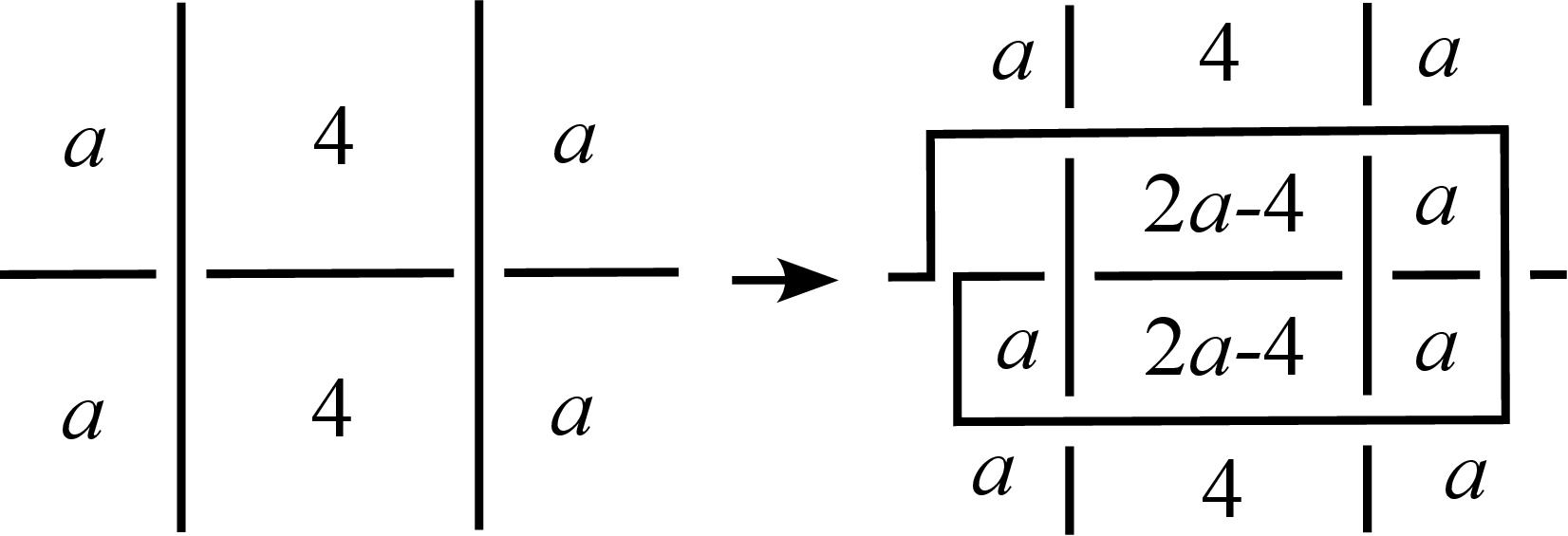}
    \caption{}
    \label{app02}
\end{minipage}
\hspace{0.04\columnwidth}
\begin{minipage}[b]{0.45\linewidth}
\centering 
\includegraphics[clip,height=2cm]{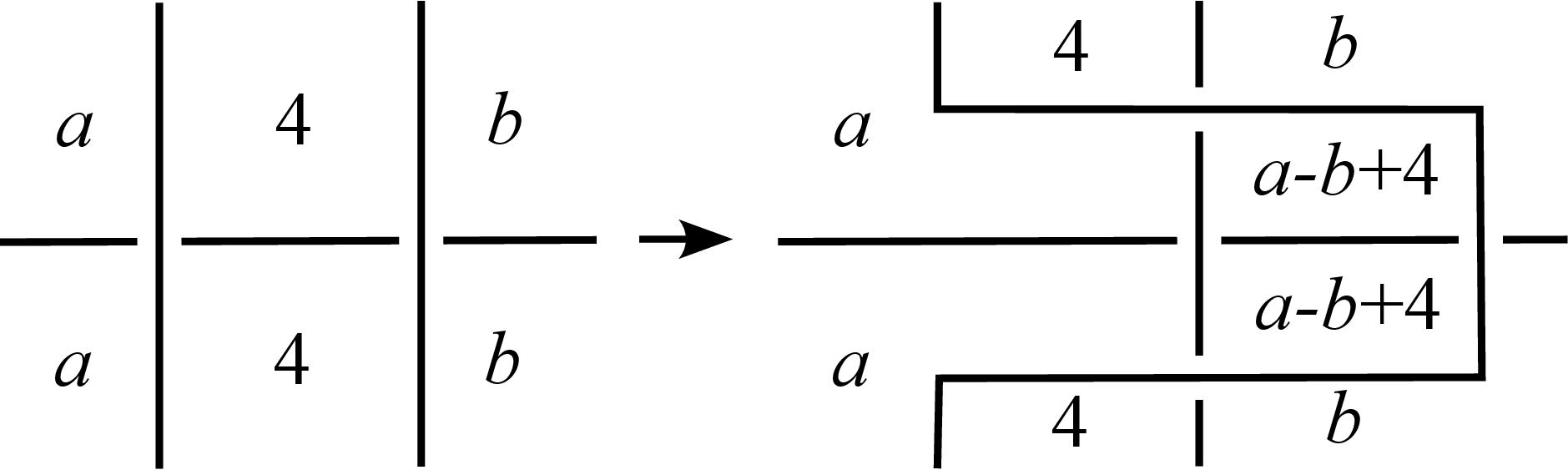}
    \caption{}
    \label{app03}
\end{minipage}    
\end{figure}
\begin{figure}[H]
\begin{minipage}[b]{0.45\linewidth}
  \centering    \includegraphics[clip,height=2cm]{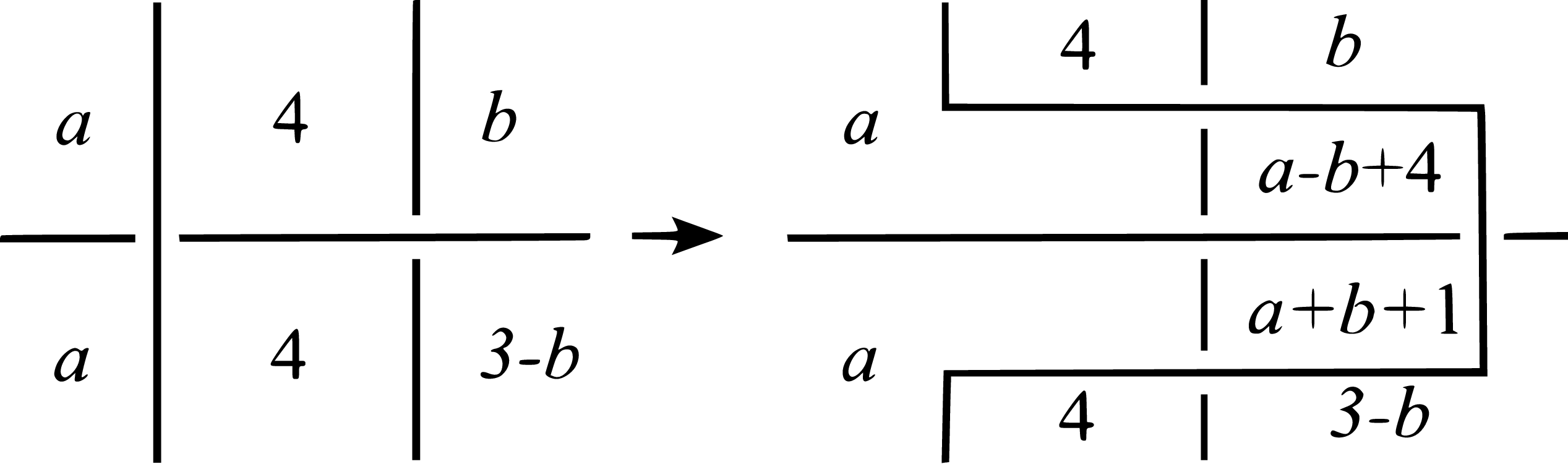}
    \caption{}
    \label{app04}
\end{minipage}
\hspace{0.08\columnwidth}
\begin{minipage}[b]{0.45\linewidth}
\centering 
\includegraphics[clip,height=2cm]{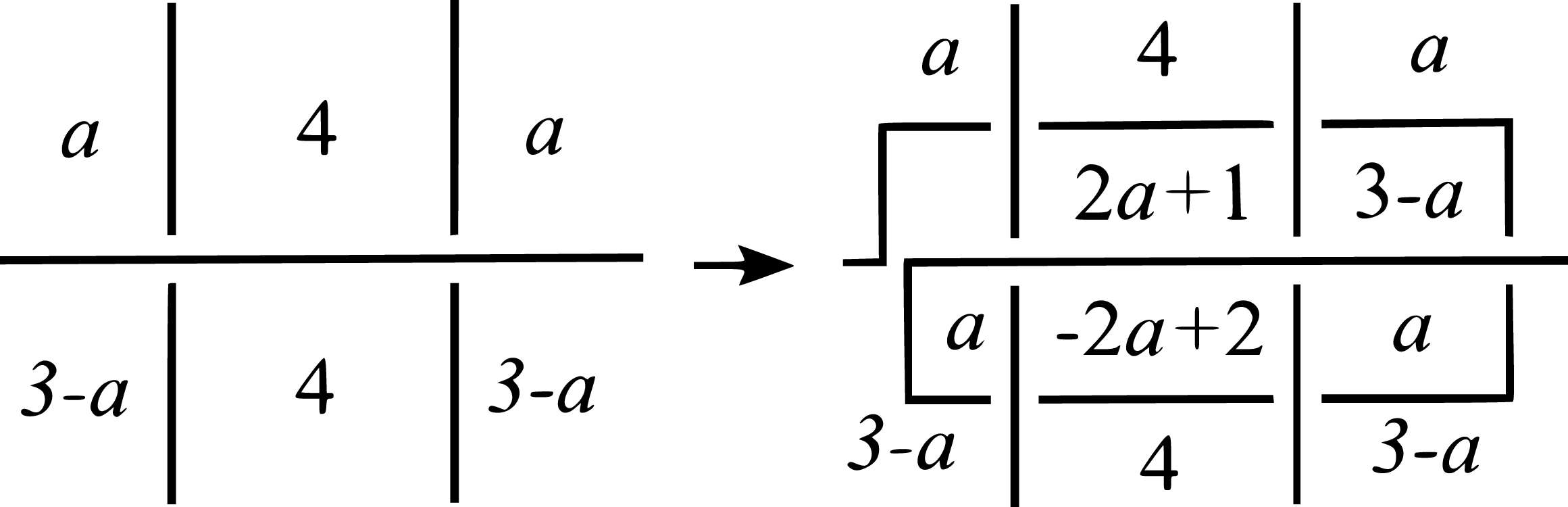}
    \caption{}
    \label{app05}
\end{minipage}    
\end{figure}
\begin{figure}[H]
\includegraphics[clip,height=2cm]{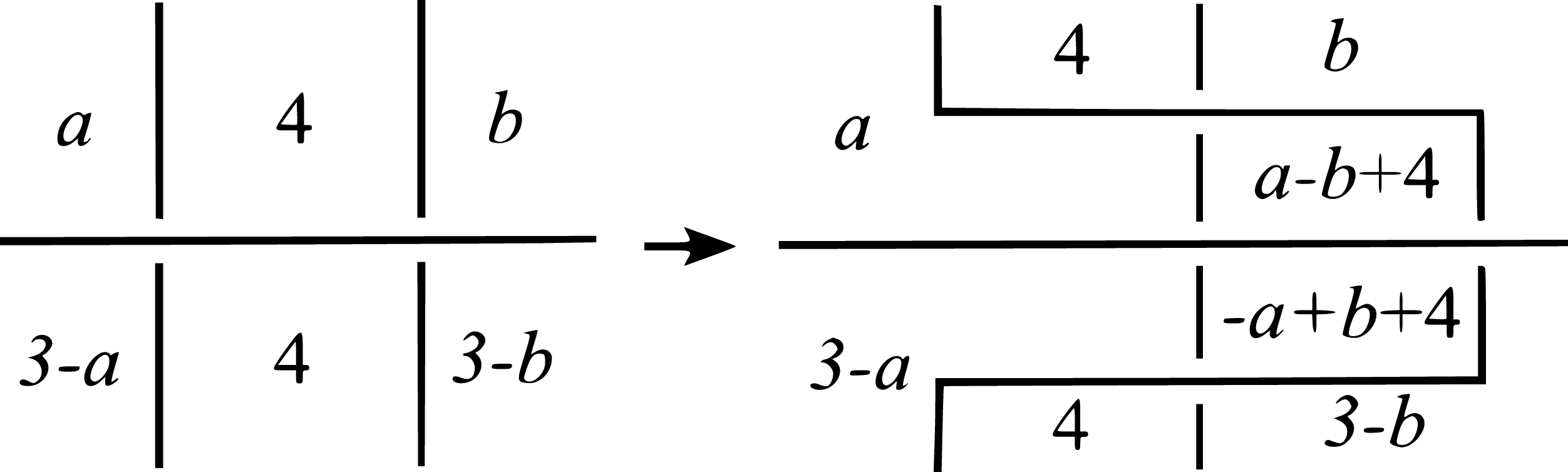}
    \caption{}
    \label{app06}
\end{figure}
Here, we note that 
\begin{align*}
&\mbox{$2a-4 \neq 4$ for Figure \ref{app02},}\\
&\mbox{$a-b+4 \neq 4$ for Figure \ref{app03},}\\
&\mbox{$3-b, a-b+4\neq 4$ for Figure \ref{app04},}\\
&\mbox{$3-a, 2a+1, -2a+2 \neq 4$ for Figure \ref{app05},}\\
&\mbox{$3-a, 3-b, a-b+4, -a+b+4 \neq 4$ for Figure \ref{app06}}    
\end{align*}
from $a\not =b$ and $a,b \not =4$.



Next we delete crossings that have two regions colored by $4$s diagonally as 
the following move depicted in Figure \ref{app07}.
Here we note that there exists a crossing which has four regions colored by $a,b,c$ and $4$ near an $\{a,4\}$-semiarc.
\begin{figure}[H]
\includegraphics[clip,height=1.5cm]{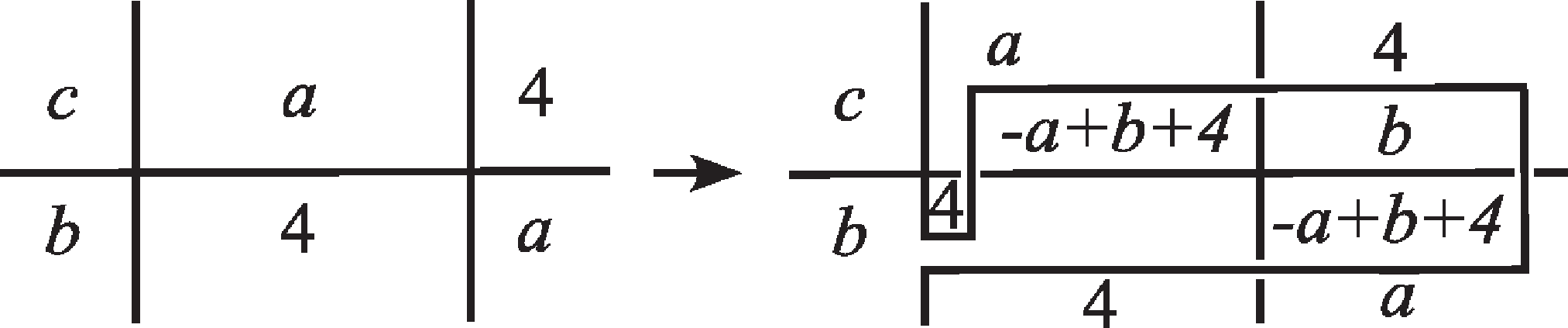}
    \caption{}
    \label{app07}
\end{figure}

Here we obtain the color $-a+b+4$ that is not equal to $4$ from $a\neq b$. 
Thus, for the resultant Dehn $5$-colored diagram, there exist no crossings that have four or two regions colored by $4$. \\

(Step 2) 
For any $n$-gon, say $x$, colored by $4$ with $n\geq 4$, we focus on an $\{a,4\}$-semiarc for some $a\not =4$, say $u$, on the boundary of $x$. 
As in Figure \ref{app11}, we move the semiarc $u$ by RII-moves with each semiarc on the boundary of $x$ which is not an $\{a,4\}$-semiarc, 
where we perform the RII-moves as $u$ goes under each semiarc. 
Here, the 2-gons with no label in the right of Figure \ref{app11} are not colored by $4$. Thus, we can delete the $n$-gon with the color $4$.
\begin{figure}[H]
\includegraphics[clip,height=2.5cm]{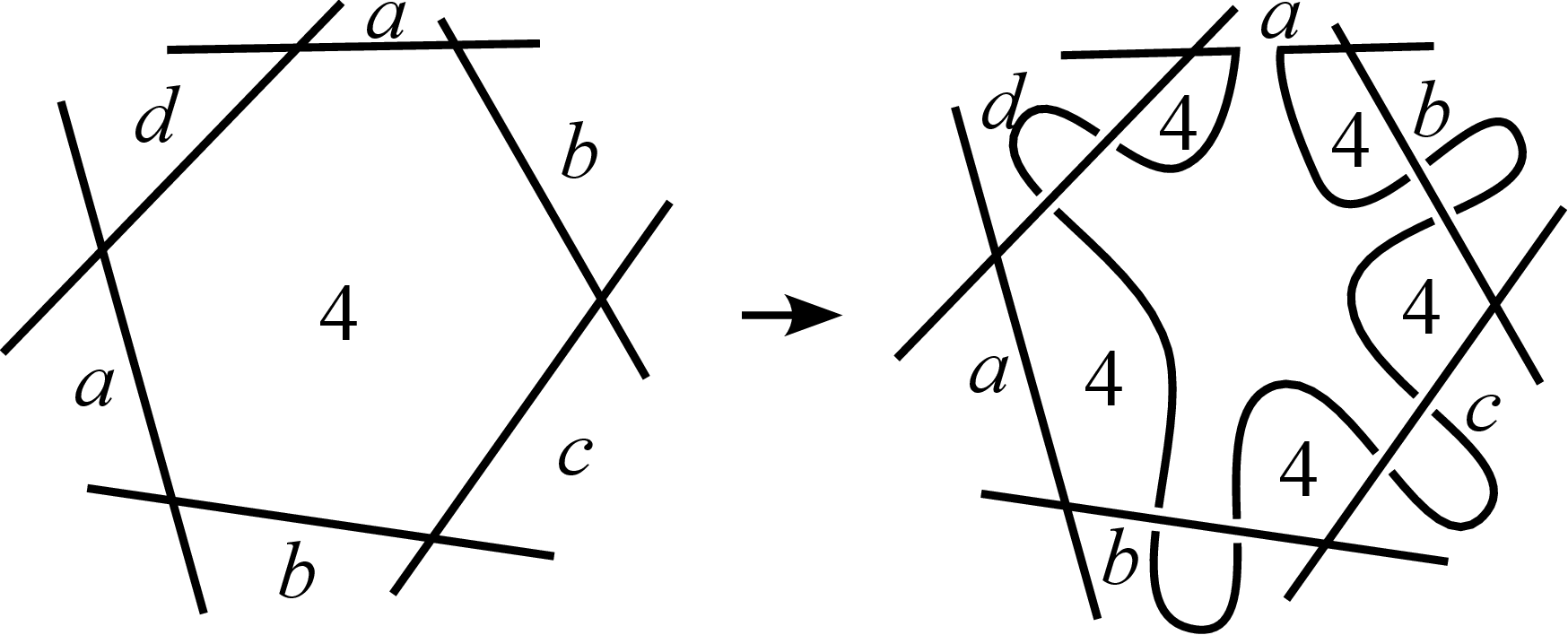}
    \caption{}
    \label{app11}
\end{figure}
We note that the obtained new regions colored by $4$ are $4$-, $3$-, $2$-gons such that the color $4$ is not appeared in the neighborhood of the regions. We also note that each $4$-gon colored by $4$ is as depicted in the left of Figure~\ref{app12}.

Next we delete $4$-gons colored by $4$ as depicted in Figure~\ref{app12} if $b\neq c$, Figure~\ref{app13} if $b=c,a\not=2b+1$ and Figure~\ref{app14} if $b=c,a=2b+1$.


\begin{figure}[H]
\includegraphics[clip,height=1.5cm]{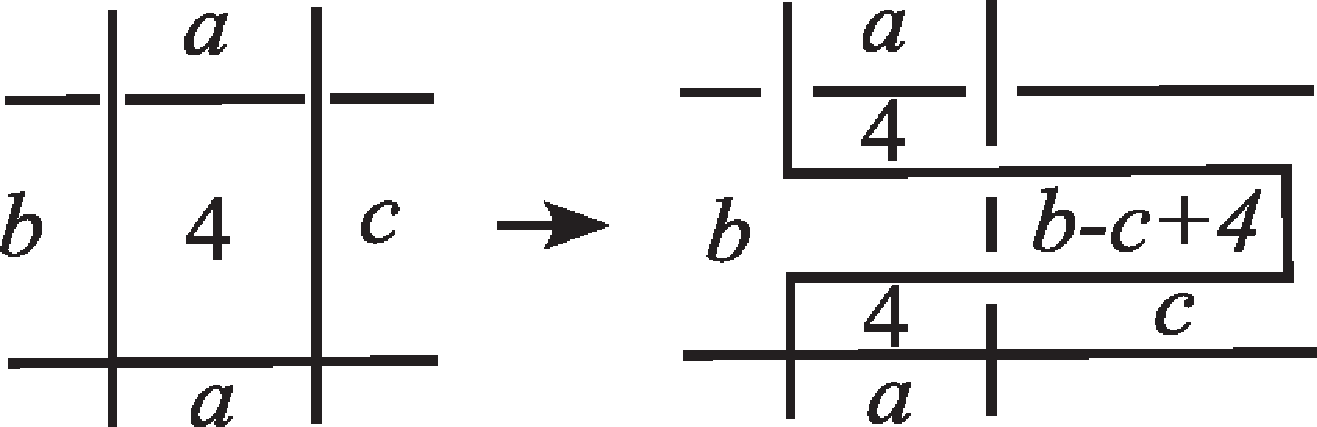}
    \caption{}
    \label{app12}
\end{figure}


\begin{figure}[H]
\includegraphics[clip,height=3.3cm]{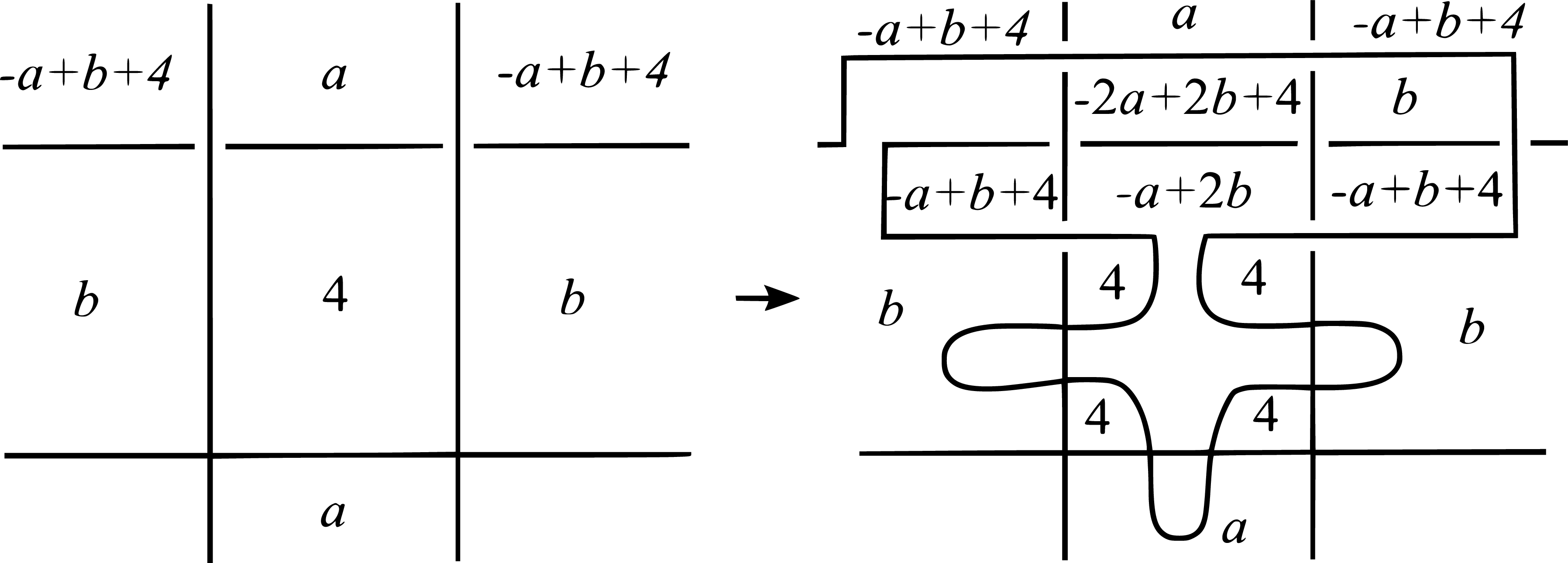}
    \caption{}
    \label{app13}
\end{figure}
\begin{figure}[H]
\includegraphics[clip,height=4cm]{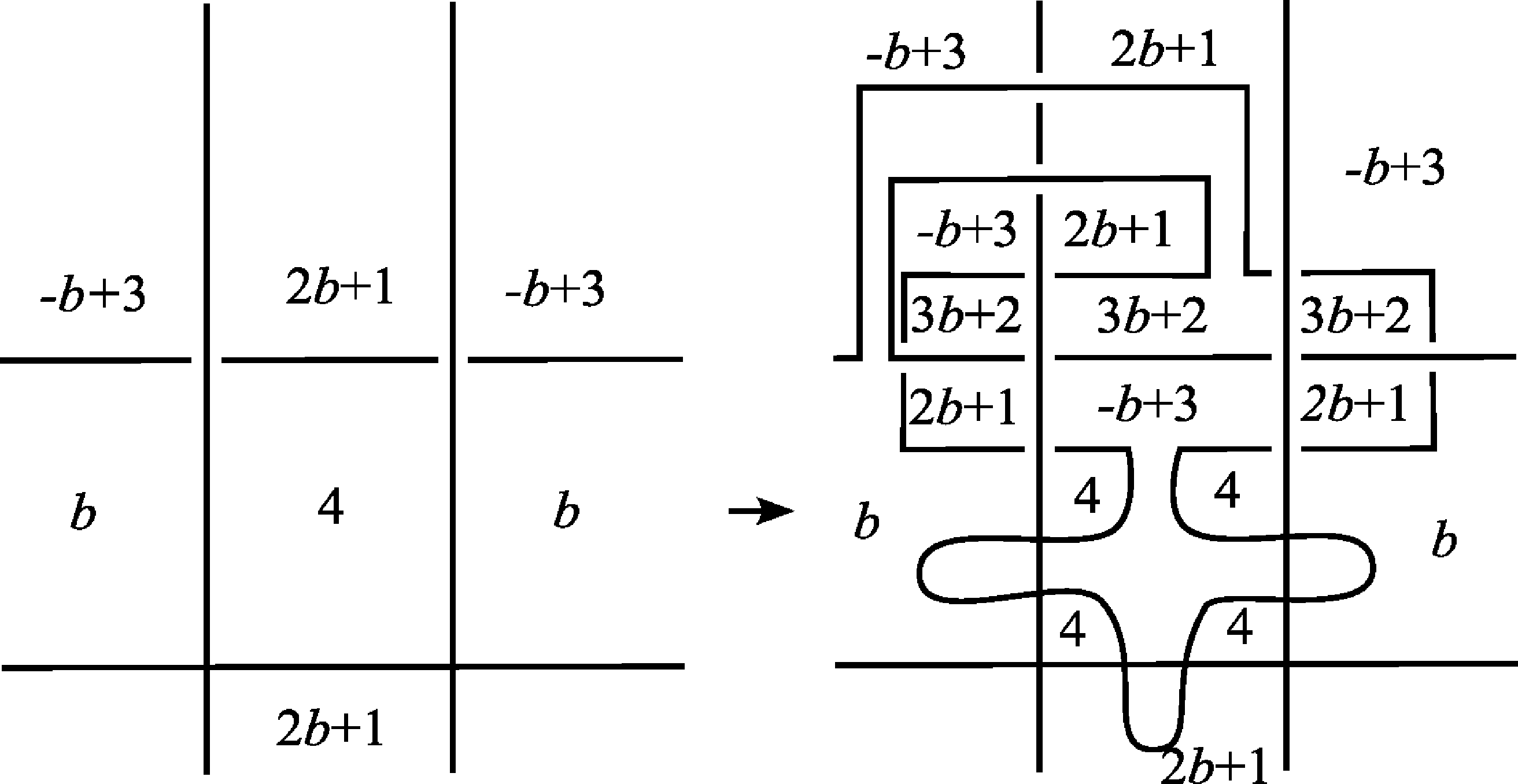}
    \caption{}
    \label{app14}
\end{figure}
Here, we note that
\begin{align*}
&b-c+4 \neq 4 \mbox{ for Figure~\ref{app12}},\\
&-a+b+4, -a+2b, -2a+2b+4 \neq 4 \mbox{ for Figure~\ref{app13}, and}\\
&-b+3, 2b+1, 3b+2\neq 4 \mbox{ for Figure~\ref{app14}}
\end{align*}
from $b \not =c$, $a\not=b$; $a\not =2b+1$, and $b\not= 4$,
respectively, and the 2-gons with no label in the right of Figures \ref{app13} and \ref{app14} are not colored by $4$.


Thus, for the resultant Dehn $5$-colored diagram, each region colored by $4$ is a $3$- or $2$-gon around which the color $4$ is not appeared.\\



(Step 3) 
Now we consider $3$-gons colored by $4$, where we note that such $3$-gons are as depicted in the left of Figure~\ref{app22}, the left of Figure~\ref{app23} or the mirror images for some distinct $a,b,c$. 
 
We first remove $3$-gons as depicted in the left of Figure \ref{app22} by the move shown in Figure \ref{app22}.
\begin{figure}[H]
  \centering    \includegraphics[clip,height=2.5cm]{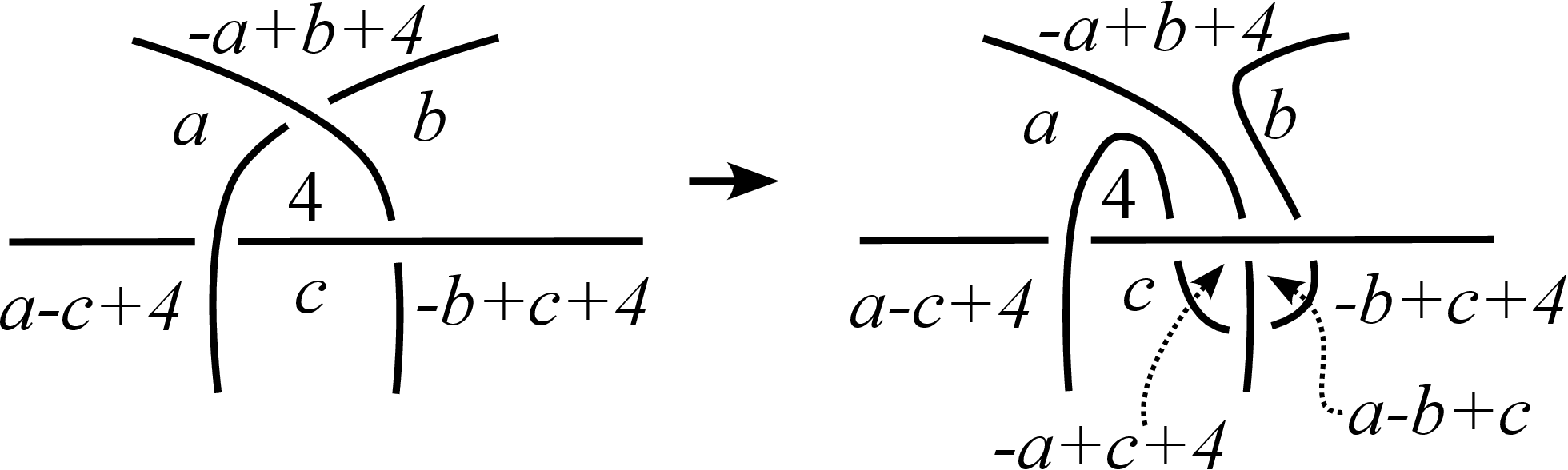}
    \caption{}
    \label{app22}
\end{figure}
Here, we note that
\begin{eqnarray*}
-a+b+4, -b+c+4, a-c+4, -a+c+4 \not = 4
\end{eqnarray*}
because $a,b,c$ are distinct.
Although the new $3$-gon labeled by $a-b+c$ might be colored by $4$, by considering the case of 3-gons in the left of Figure~\ref{app23} or the mirror image, the $3$-gon with the color $4$ can be removed.
As for the mirror images of the $3$-gons in the left of Figure \ref{app22}, they are also removed by taking the mirror of the same move. 

Next we remove $3$-gons as depicted in the left of Figure \ref{app23} by the move shown in Figure \ref{app23} if $a-b+c\not =4$ and Figure \ref{app23-1} if $a-b+c =4$.
\begin{figure}[H]
\centering 
\includegraphics[clip,height=2.2cm]{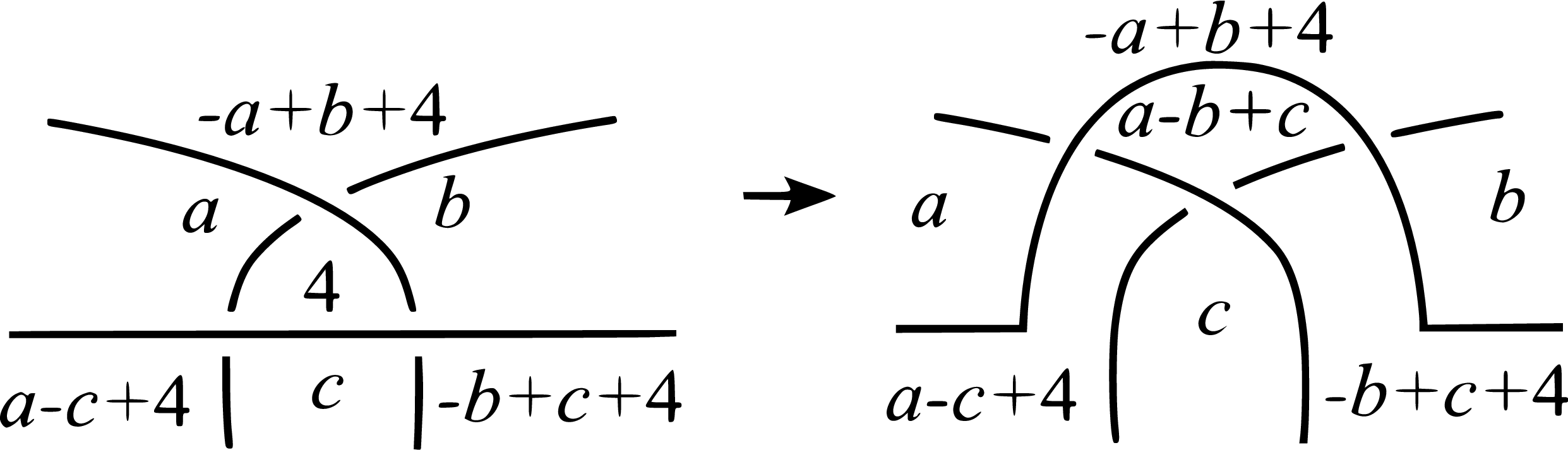}
    \caption{}
    \label{app23} 
\end{figure}

\begin{figure}[H]
\centering 
\includegraphics[clip,height=2.2cm]{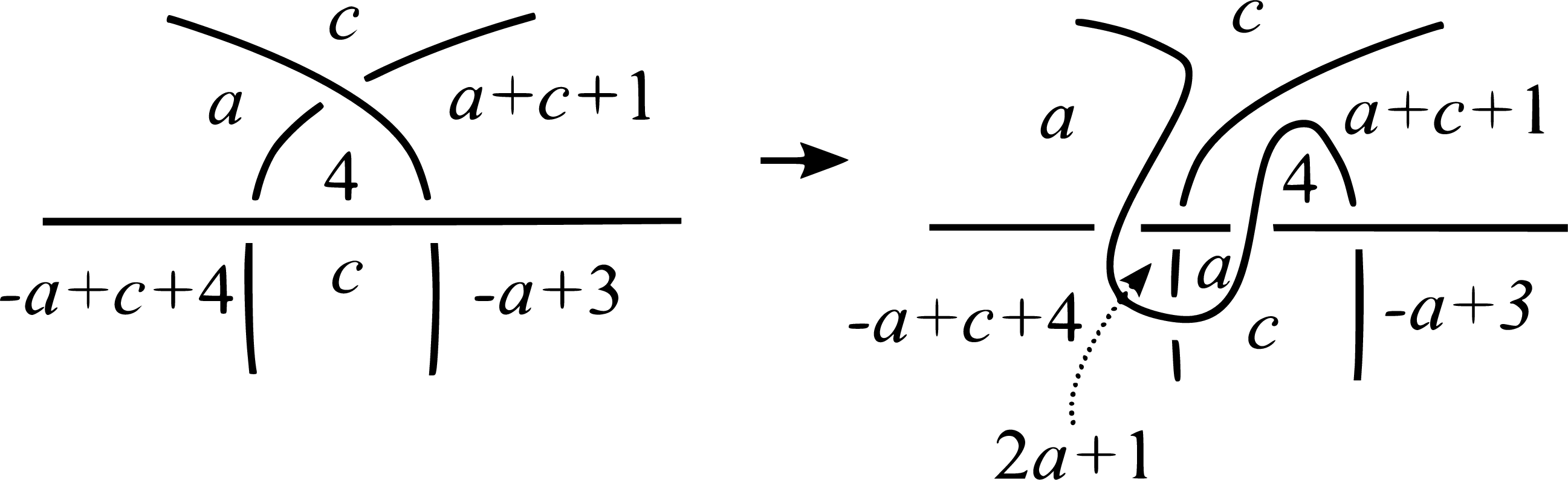}
    \caption{}
    \label{app23-1} 
\end{figure}

Here, we note 
\begin{eqnarray*}
-a+b+4, -b+c+4, -a+c+4 \neq 4 \mbox{ for Figure~\ref{app23}}, \\
-a+c+4, a+c+1, -a+3, 2a+1 \neq 4 \mbox{ for Figure~\ref{app23-1}}
\end{eqnarray*}
because $a,b,c$ are distinct and $a,b,c\neq 4$.
As for the mirror images of the $3$-gons in the left of Figure \ref{app23}, they are also removed by taking the mirror of the same moves.

Next let us consider the case of $2$-gons colored by $4$, where we note that such $2$-gons are as depicted in the left of Figure~\ref{app24}, the left of Figure~\ref{app25} or the mirror image for some $a,b$.

The $2$-gons as depicted in the left of Figure~\ref{app24} can be deleted by the move in Figure~\ref{app24}. 
\begin{figure}[H]
  \centering    \includegraphics[clip,height=1.5cm]{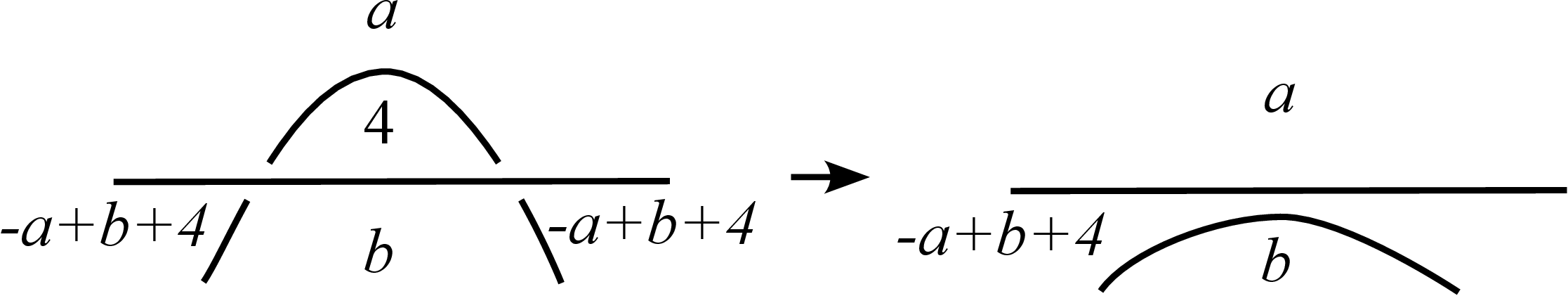}
    \caption{}
    \label{app24}
\end{figure}
Here, note that
\begin{eqnarray*}
-a+b+4\not =4
\end{eqnarray*}
from $a\neq b$.
As for the $2$-gons as depicted in the left of Figure~\ref{app25}, we remove them 
by the move as shown in Figure~\ref{app25} if $-a+2b\not =4$ and $-2a+3b\not=4$, 
Figure~\ref{app26} if $-a+2b=4$ (and $-2a+3b\not=4$), and Figure~\ref{app27} if $-2a+3b=4$ (and $-a+2b\not=4$).



\begin{figure}[H]
\centering 
\includegraphics[clip,height=2cm]{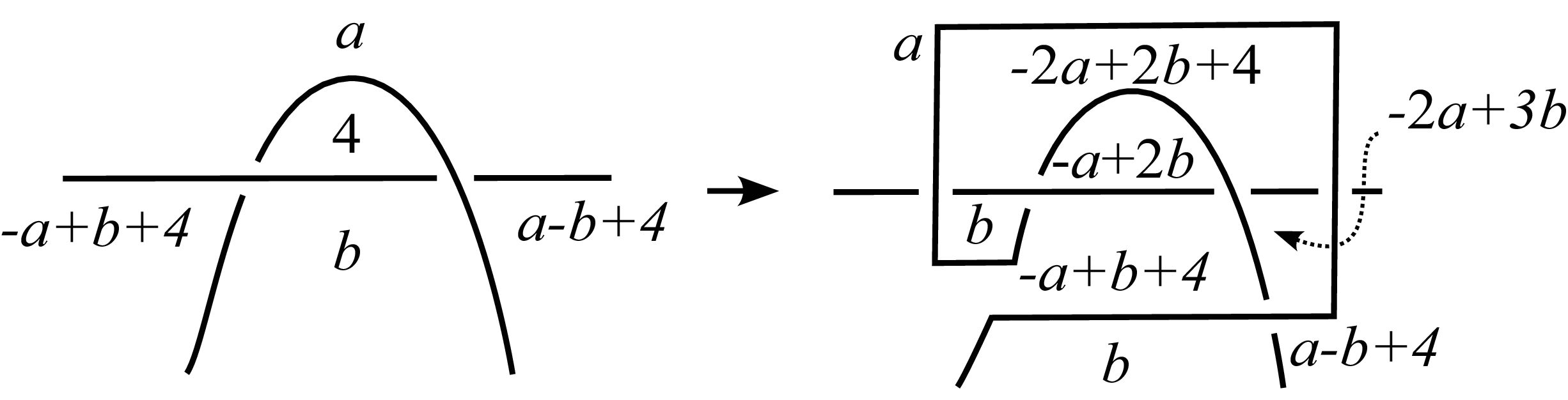}
    \caption{}
    \label{app25}
\end{figure}


\begin{figure}[H]
  \centering    \includegraphics[clip,height=2cm]{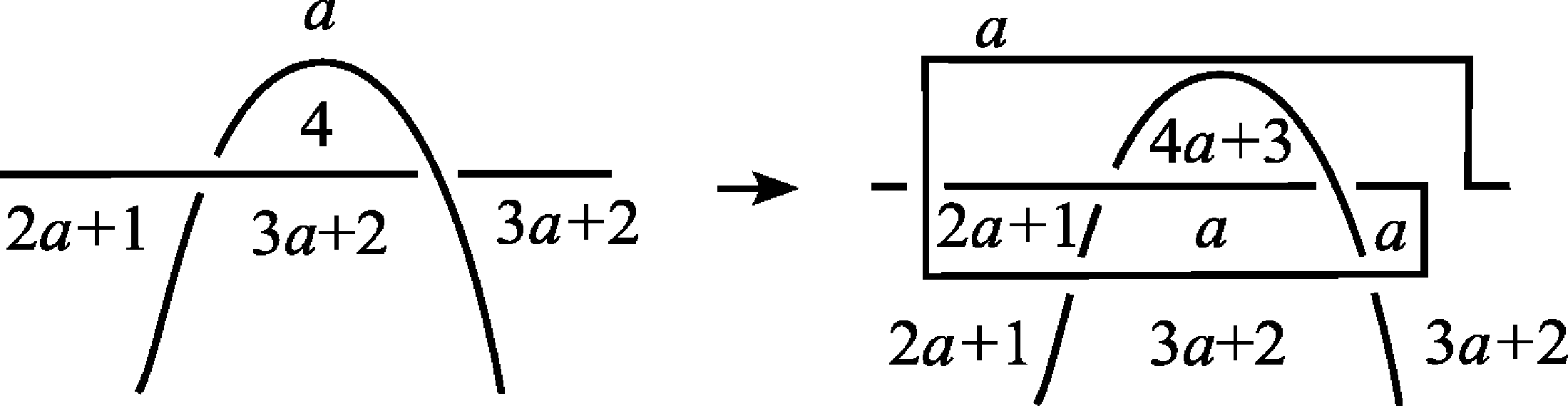}
    \caption{}
    \label{app26}
\end{figure}


\begin{figure}[H]
\centering 
\includegraphics[clip,height=2cm]{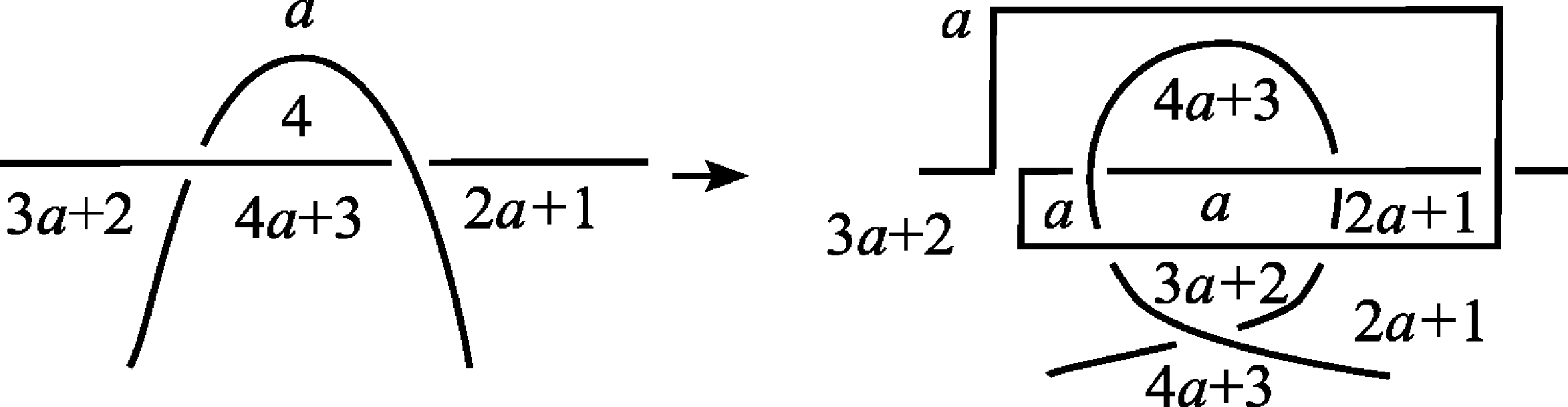}
    \caption{}
    \label{app27}
\end{figure}

Here, we note that
\begin{align*}
&-a+b+4, a-b+4, -2a+2b+4 \neq 4 \mbox{ for Figure~\ref{app25}},\\
&2a+1, 3a+2, 4a+3 \neq 4 \mbox{ for Figure~\ref{app26}, and}\\
&3a+2, 4a+3, 2a+1 \neq 4 \mbox{ for Figure~\ref{app27}}\\
\end{align*}
because $a,b,c$ are distinct and $a,b,c\neq 4$. 
As for the mirror images of the $2$-gons in the left of Figure \ref{app25}, they are also removed by taking the mirror of the same moves.

The resultant diagram has no region colored by $4$.
This completes the proof.
\end{proof}

\section*{Acknowledgments}
The authors wish to express their thanks to Professor Shin Satoh for several helpful comments on Theorem~\ref{th:palette} and the property $\mincol_5(K)=4$.
The second author was supported by JSPS KAKENHI Grant Number 21K03233.


\begin{thebibliography}{000}
\bibitem{AbchirElhamdadiLamsifer}
H.~Abchir, M.~Elhamdadi and S.~Lamsifer, 
\textit{On the minimum number of Fox colorings of knots}, 
 Grad. J. Math. {\bf 5} (2020), no.~2, 122--137. 

\bibitem{BentoLopes}
F.~Bento and P.~Lopes, 
\textit{The minimum number of Fox colors modulo 13 is 5},
Topology Appl. {\bf 216} (2017), 85--115.

\bibitem{CrowellFox}
R.~H.~Crowell and  R.~H.~Fox, {\it Introduction to knot theory},   Graduate Texts in Mathematics, No. 57. Springer-Verlag, New York-Heidelberg, 1977. x+182 pp.


\bibitem{HanZhou}
Y.~Han and B.~Zhou, 
\textit{The minimum number of coloring of knots}, 
J. Knot Theory Ramifications {\bf 31} (2022), no.~2, Paper No. 2250013, 55 pp.

\bibitem{HararyKauffman}
F.~Harary and L.~H.~Kauffman, 
\textit{Knots and graphs. I. Arc graphs and colorings}, 
Adv. Appl. Math. {\bf 22}(3) (1999), 312--337.

\bibitem{IchiharaMatsudo}
K. Ichihara\ and\ E. Matsudo,
 A lower bound on minimal number of colors for links, 
 Kobe J. Math. {\bf 33} (2016), no.~1-2, 53--60.

\bibitem{MatsudoOshiroYamagishi-2}
E.~Matsudo, K.~Oshiro and G.~Yamagishi, 
\textit{Minimum numbers of Dehn colors of knots and symmetric local biquandle cocycle invariants}, 
arXiv:2501.09942.



\bibitem{NakamuraNakanishiSatoh13}
T.~Nakamura, Y.~Nakanishi, and S.~Satoh, 
\textit{The pallet graph of a Fox coloring}, 
Yokohama Math. J. {\bf 59} (2013), 91--97.

\bibitem{NakamuraNakanishiSatoh16}
T.~Nakamura, Y.~Nakanishi, and S.~Satoh, 
\textit{11-colored knot diagram with five colors}, 
J. Knot Theory Ramifications {\bf 25} (2016), no.~4, Paper No. 1650017, 22 pp.



\bibitem{Oshiro10}
K.~Oshiro, 
\textit{Any 7-colorable knot can be colored by four colors},
J. Math. Soc. Japan {\bf 62}(3) (2010) 963--973.





\bibitem{Satoh09}
S.~Satoh, 
\textit{5-colored knot diagram with four colors},  
Osaka J. Math. {\bf 46} (2009), no.~4, 939--948.


\end{thebibliography}
\end{document}